\documentclass[11pt,a4paper]{article}

\usepackage[colorlinks=true, pdfstartview=FitV, linkcolor=blue, citecolor=blue, urlcolor=blue]{hyperref}

\usepackage{amsmath}
\usepackage{amsthm}
\usepackage{amssymb}
\usepackage{mathtools}
\usepackage{fullpage}

\usepackage[margin=10pt,font=small,labelfont=bf]{caption}
\usepackage{graphicx}
\usepackage{tikz}
\usepackage{pgfplots}
\usepgfplotslibrary{groupplots}
\pgfplotsset{compat=1.10}
\pgfplotsset{
    tick align=outside,
    x grid style={white},
    xmajorgrids,
    y grid style={white},
    ymajorgrids,
    axis line style={white},
    axis background/.style={fill=white!92!black},
    legend style={draw=white, fill=white},
    legend cell align={left}
}

\usepackage[capitalise]{cleveref}

\newtheorem{thm}{Theorem}
\newtheorem{proposition}[thm]{Proposition}
\newtheorem{lemma}[thm]{Lemma}

\theoremstyle{remark}

\newcommand{\ip}[2]{\langle #1,#2\rangle} 
\newcommand{\ipr}[2]{\big\langle #1,#2\big\rangle_\rho} 

\newcommand{\pa}{\partial}
\newcommand{\eps}{\epsilon}

\newcommand{\R}{\mathbb{R}}
\newcommand{\C}{\mathbb{C}}
\newcommand{\bfu}{\mathbf{u}}

\newcommand{\Ai}{{\rm Ai}}

\def\Reel{{\cal R}e \,}
\def\Re{{\cal R}e \,}
\def\Im{{\cal I}m \,}

\def\sspace{\smallskip \noindent}
\def\mspace{\medskip \noindent}

\newcommand{\OO}{\mathcal{O}}

\newcommand{\dd}{\mathrm{d}} 

\newcommand{\tnorm}[1]{{\left\vert\kern-0.25ex\left\vert\kern-0.25ex\left\vert #1
    \right\vert\kern-0.25ex\right\vert\kern-0.25ex\right\vert}}
\newcounter{savecntrP7}
\newcounter{restorecntrP7}

\title{On the ill-posedness of the triple deck model}
\author{Helge Dietert%
  \setcounter{savecntrP7}{\value{footnote}}
  \thanks{Universit\'e de Paris and Sorbonne Universit\'e, CNRS,
    Institut de Math\'ematiques de Jussieu-Paris Rive Gauche (IMJ-PRG),
    F-75013, Paris, France}\;
  \thanks{Currently on leave and working at
    Institut f\"ur Mathematik, Universit\"at Leipzig, D-04103 Leipzig, Germany}
  \and David G\'erard-Varet
  \setcounter{restorecntrP7}{\value{footnote}}%
  \setcounter{footnote}{\value{savecntrP7}}\footnotemark
  \setcounter{footnote}{\value{restorecntrP7}}\; \thanks{Institut Universitaire de France, F-75205 Paris, France}
}
\begin{document}
\maketitle

\begin{abstract}
  We analyze the stability properties of the so-called triple deck
  model, a classical refinement of the Prandtl equation to describe
  boundary layer separation.  Combining the methodology introduced in
  \cite{MR3835243}, based on complex analysis tools, and stability
  estimates inspired from \cite{MR3925144}, we exhibit unstable
  linearizations of the triple deck equation. The growth rates of the
  corresponding unstable eigenmodes scale linearly with the tangential
  frequency. This shows that the recent result of Iyer and Vicol
  \cite{IyVi} of local well-posedness for analytic data is essentially
  optimal.
 \end{abstract}

\section{Introduction}
Our concern in this paper is the triple deck model, introduced in the
1960's to describe the so-called boundary layer separation. The
general concern behind this model is to understand the behaviour of
Navier-Stokes solutions with velocity $\bfu_\nu = (u_\nu, v_\nu)$ and
pressure $p_\nu$ near a rigid boundary, when the inverse Reynolds
number $\nu$ goes to zero. Due to the no-slip condition at the
boundary, it is well-known that this is a singular asymptotic problem:
the Euler solution $\bfu_0 = (u_0, v_0)$ does not describe the
dynamics near the wall. In a celebrated paper \cite{Prandtl04}, Ludwig
Prandtl tackled this problem through the use of matched asymptotic
expansions.  For planar flows in the half plane
$\Omega = \R \times \R_+$, this means that two regions should be
distinguished: one away from the wall, where
\begin{equation}
\bfu_\nu(T,X,Y)  \approx \bfu_0(t,X,Y), \quad p_\nu(T,X,Y)  \approx p_0(t,X,Y)
\end{equation}
while close to the wall,  in a boundary layer, one should have
\begin{equation} \label{P_exp}
\begin{aligned}
 u_\nu(T,X,Y)  & \approx u_P\left(T,X,\frac{Y}{\sqrt{\nu}}\right), \quad v_\nu(T,X,Y)  \approx \sqrt{\nu} \, v_P\left(T,X,\frac{Y}{\sqrt{\nu}}\right),  \\
 p_\nu(T,x,Y)   & \approx p_P\left(T,X,\frac{Y}{\sqrt{\nu}}\right)
 \end{aligned}
\end{equation}
for boundary layer profiles $(u_P,v_P,p_P)  = (u_P,v_P,p_P)(T,X,y)$.  Moreover, by injecting the Prandtl boundary layer expansion in the Navier-Stokes equation and keeping the leading order terms, we end up with the Prandl system
\begin{equation} \label{P} \tag{P}
 \begin{aligned}
 \pa_T u_P + u_P \pa_X u_P + v_P \pa_y u_P - \pa^2_y u_P + \pa_X p_P & = 0, \\
 \pa_y p_P & = 0, \\
 \pa_X u_P + \pa_y v_P & = 0, \\
 u_P\vert_{y=0} = v_P\vert_{y=0} & =  0.
 \end{aligned}
\end{equation}
This system is completed by the conditions at infinity
$$ \lim_{y \rightarrow \infty} u_P(T,X,y) = u_0(T,X,0), \quad  \lim_{y \rightarrow \infty} p_P(T,X,y) = p_0(T,X,0) $$
which ensure the matching between the boundary layer and the upper inviscid region of the flow.

\mspace
The Prandtl model has revealed very fruitful to understand steady
Navier-Stokes flows in regions where boundary layers remain attached
to the boundary. However, it is well-known that downstream of the
flow, under an adverse pressure gradient, streamlines detach from the
boundary and recirculation occur. Moreover, in the unsteady context,
even upstream, Tollmien-Schlichting instabilities may destabilize the
flow. All these hydrodynamic phenomena have consequences on the
mathematical analysis of system \eqref{P}, for which various negative
results have been obtained: ill-posedness results
\cite{MR2601044,LiuYang}, blow-up results
\cite{MR1476316,gargano2,MR3590519,CoMa}, instability of Prandtl
expansions at the level of the Navier-Stokes equations
\cite{Gre,MR3566199,MR4038143}.  A common difficulty behind these
works is the appearance of small tangential scales, that invalidate
expansions of type \eqref{P_exp}, which are assumed to depend
regularly on $x$. In order to capture the effect of these small
scales, while still trying to obtain reduced models, several
refinements of the Prandtl model were introduced in the 1960's and
1970's. The most famous ones are the triple deck model and the
Interactive Boundary Layer model (IBL). The latter one, analyzed
mathematically in the recent paper \cite{MR3835243}, consists in
keeping additional $O(\sqrt{\nu})$ terms, resulting in a coupling
between the inviscid equations for the upper region, and the
(modified) Prandtl equation.

Here we focus on the triple deck model. We first extend the derivation
given in Lagrée \cite{Lag} to the unsteady setting. The basic idea is
to study perturbations to the main Prandtl flow, in the vicinity of
$T = T^*, X = X^*$ (typically the time and abscissa of separation),
with small scale variations in $T,X$. Denoting $\eps$ the amplitude of
the perturbation, and $\eta$, $\delta$ the small time and tangential
scales, we write
\begin{equation*}
\begin{aligned}
  u_\nu(T,X,Y) & \approx u_P\left(T,X,\frac{Y}{\sqrt{\nu}}\right)
  + \eps  \tilde u\left(\frac{T-T^*}{\eta},\frac{X-X^*}{\delta},\frac{Y}{\sqrt{\nu}}\right) \\
  & \approx u_P\left(T^*,X^*,\frac{Y}{\sqrt{\nu}}\right)
  + \eps  \tilde u\left(\frac{T-T^*}{\eta},\frac{X-X^*}{\delta},\frac{Y}{\sqrt{\nu}}\right)  + \OO(\eta) + \OO(\delta) \\
 v_\nu(T,X,Y) & \approx \sqrt{\nu} \,
 v_P\left(T,X,\frac{Y}{\sqrt{\nu}}\right)
 + \sqrt{\nu} \, \frac{\eps}{\delta}
 \tilde v\left(\frac{T-T^*}{\eta},\frac{X-X^*}{\delta},\frac{Y}{\sqrt{\nu}}\right) \\
 & \approx \sqrt{\nu} \, \frac{\eps}{\delta}
 \tilde v\left(\frac{T-T^*}{\eta},\frac{X-X^*}{\delta},\frac{Y}{\sqrt{\nu}}\right)  +  \OO(\sqrt{\nu}) \\
 p_\nu(T,X,Y) & \approx p_P(T,X)
 + \eps^2 \tilde p\left(\frac{T-T^*}{\eta},\frac{X-X^*}{\delta},\frac{Y}{\sqrt{\nu}}\right)
 \end{aligned}
\end{equation*}
(we anticipate that the amplitude of the pressure is $\eps^2$, see
below).  Injecting the ansatz into the Navier-Stokes equations, we
derive the relations satisfied by
$(\tilde u, \tilde v, \tilde p) = (\tilde u, \tilde v, \tilde
p)(t,x,y)$. With notation $U(y) := u_P(T^*,X^*,y)$, anticipating that
$\eta \gg \delta$ and \(\nu \ll \epsilon \delta^2\), we get
\begin{equation*}
  \pa_x \tilde u + \pa_y \tilde v = 0, \quad U \pa_x \tilde u + U'
  \tilde v = 0, \quad \pa_y \tilde p = 0.
\end{equation*}
The second identity reads
$U^2 \pa_y \left( \frac{\tilde v}{U} \right) = 0$. Thanks to this
relation and to the divergence-free condition, we can introduce a
function $A = A(t,x)$ such that
\begin{equation*} \tilde u(t,x,y) = A(t,x) U'(y), \quad \tilde
  v(t,x,y) =  - \pa_x A(t,x) U(y).
\end{equation*}
In particular, we see that $\tilde u(t,x,0) = A(t,x) U'(0)$ is
non-zero. To restore the no-slip condition, one must add a
sublayer. This sublayer is referred to as the {\em lower deck}, while
the main one, corresponding originally to the Prandtl layer, is the
{\em main deck}. Eventually, the upper region outside the
$\OO(\sqrt{\nu})$ boundary layer is called the {\em upper deck}. Let
$h$ be the typical length scale of the lower deck, and $z = y/h$. The
velocity at the bottom of the main deck reads
$U(y) + \eps \tilde u(t,x,y) \approx h U'(0) z + \eps \tilde
u(t,x,0)$. For matching between the lower and main deck, it is
therefore natural to take $h = \eps$, and to look for an asymptotics
in the lower deck of the form:
\begin{equation*}
\begin{aligned}
u_\nu(T,X,Y)  & \approx \eps u\left(\frac{T-T^*}{\eta},\frac{X-X^*}{\delta},\frac{Y}{\sqrt{\nu} \eps}\right), \\
v_\nu(T,X,Y)  & \approx \sqrt{\nu} \frac{\eps^2}{\delta}  v\left(\frac{T-T^*}{\eta},\frac{X-X^*}{\delta},\frac{Y}{\sqrt{\nu} \eps}\right), \\
p_\nu(T,X,Y) & \approx \eps^2 p\left(\frac{T-T^*}{\eta},\frac{X-X^*}{\delta},\frac{Y}{\sqrt{\nu} \eps}\right)
\end{aligned}
\end{equation*}
with $ (u,v) = (u,v)(t,x,z)$. Moreover, in order to match the effects
of time variation, advection, and diffusion in the lower deck,
$\pa_t \sim u \pa_X \sim \nu \pa^2_Y$, one has to take
$\delta \sim \eps^3$, $\eta \sim \eps^2$. The amplitude $O(\eps^2)$ of
the pressure term allows to retain it as well. This results in
 \begin{equation*}
 \begin{aligned}
 \pa_t u + u \pa_x u + v \pa_z u - \pa^2_z u + \pa_x p & = 0, \\
 \pa_z p & = 0, \\
 \pa_x u + \pa_z v & = 0, \\
 u\vert_{z=0} = v\vert_{z=0} & =  0.
 \end{aligned}
\end{equation*}
These equations are the same as those in \eqref{P}. But, the boundary conditions at infinity differ from the classical ones. Assume $U(\infty) = 1$, $U'(0) =1$  for simplicity. On one hand, matching of the velocities of the lower and main desks  yields
\begin{equation*}
u(t,x,z) \sim U'(0) z + \tilde u(t,x,0) = z + A(t,x) , \quad z \rightarrow +\infty
\end{equation*}
On the other hand, as explained in \cite{Lag} and apparent in the
original Prandtl layer \eqref{P}, the $\OO(\eps^2)$ pressure should
not change across the lower and main decks, and coincide with the
trace of the pressure in the upper deck. In this upper deck, the
dynamics is driven by the so-called blowing velocity, that is the
normal component coming from the main deck:
$\sqrt{\nu} \frac{\eps}{\delta} \, \tilde v(t,x,\infty) \sim
\frac{\sqrt{\nu}}{\eps^2}$. Anticipating that the upper deck must have
the same amplitude, we find $\eps^2 \sim \frac{\sqrt{\nu}}{\eps^2}$,
that is $\eps = \nu^{1/8}$. Finally, in the upper deck, one looks for
an asymptotics isotropic in $X,Y$ of the form
\begin{align*}
 u_\nu & \approx 1 + \nu^{1/4}  \overline{u}\left(\frac{T-T^*}{ \nu^{1/4}},\frac{X-X^*}{\nu^{3/8}},\frac{Y}{\nu^{3/8}}\right),  \quad v_\nu \approx \nu^{1/4}  \overline{v}\left(\frac{T-T^*}{\nu^{1/4}},\frac{X-X^*}{\nu^{3/8}},\frac{Y}{\nu^{3/8}}\right), \\
  p_\nu & \approx  \nu^{1/4}   \overline{p}\left(\frac{T-T^*}{\eps^2},\frac{X-X^*}{\nu^{3/8}},\frac{Y}{\nu^{3/8}}\right),
\end{align*}
Plugging the asymptotic ansatz into the Naiver-Stokes equations yields the linearized Euler dynamics for $(\overline{u}, \overline{v}, \overline{p}) = (\overline{u}, \overline{v}, \overline{p})(t,x,\theta)$:
\begin{equation*}
  \pa_x \overline{u} + \pa_x \overline{p} = 0, \quad
  \pa_x \overline{v} + \pa_\theta \overline{p} = 0, \quad
  \pa_x \overline{u} + \pa_\theta \overline{v} = 0, \quad
  \overline{v}\vert_{\theta=0} = \tilde v(t,x,\infty) = -\pa_x A(t,x).
\end{equation*}
This system can be solved using Fourier transform in $x$ as
\begin{equation*}
  \mathcal{F} \overline{p}\vert_{\theta = 0}(\xi) = -\text{sign}(\xi) \mathcal{F} \pa_x A(t,\cdot)(\xi) = |\xi| \mathcal{F} A(t,\cdot)(\xi).
\end{equation*}
In physical variables this is
\begin{equation*}
  \overline{p}\vert_{\theta = 0} = |\pa_x| A(t,x) :=  \frac{1}{\pi} \int_{\R} \frac{\pa_x A(t,x)}{x-\xi} dx,
\end{equation*}
where the right-hand side is the Hilbert transform of $\pa_x A$ in variable $x$.

Thanks to this last condition, the triple deck model can be written
\begin{equation} \label{TD} \tag{TD}
\begin{aligned}
 \pa_t u + u \pa_x u + v \pa_z u - \pa^2_z u + \pa_x  |\pa_x| A & = 0, \\
 \pa_x u + \pa_z v & = 0, \\
 u\vert_{z=0} = v\vert_{z=0} & =  0. \\
 \lim_{z \rightarrow \infty} u - z  & = A.
 \end{aligned}
\end{equation}
The unknowns are $(u,v) = (u,v)(t,x,z)$ and $A = A(t,x)$. One must complete the system with an initial data $u\vert_{t=0} = u_0(x,z)$, consistent with the structure at infinity given by the last line of \eqref{TD}.
Note that if we let $z \rightarrow +\infty$ in the first equation, using $u = z + A + o(1)$, we obtain the redundant consistency equation:
\begin{equation} \label{consistency_eq}
 \pa_t A + A \pa_x A + \pa_x  |\pa_x| A = \pa_x \int_0^{+\infty} (u-A)\,
 \dd z.
\end{equation}
The different spatial scalings are shown in
\cref{fig:spatial-scales}. Wrapping the derivation up, the overall
idea is to consider a perturbation around a boundary layer
\((u_P,v_P)\) with small tangential scale \(\delta\) around \(X=X^*\)
and assuming that away from the lower deck (which is where \(u_P\) is
expected to loose monotonicity) the inviscid terms dominate.

\begin{figure}[htb]
  \centering
  \begin{tikzpicture}[xscale=5]
    \fill[black!10!white] (1,0) -- (3,0) -- (3,4) -- (1,4);
    \draw[->] (0.5,0) -- (3.5,0) node[anchor=north] {\(X\)};
    \draw[->] (0.5,0) -- (0.5,4) node[anchor=east] {\(Y\)};
    \draw[dotted] (3.5,1) -- (0.5,1) node [anchor=east] {\(\epsilon\, \sqrt{\nu}\)};
    \draw[dotted] (3.5,3) -- (0.5,3) node [anchor=east] {\(\sqrt{\nu}\)};
    \draw (1,4) -- (1,0) node [anchor=north] {\(X^*-\delta\)};
    \draw (2,4) -- (2,0) node [anchor=north] {\(X^*\)};
    \draw (3,4) -- (3,0) node [anchor=north] {\(X^*+\delta\)};
    \draw (1,0.5) node [anchor=west] {lower deck};
    \draw (2,0.5) node [anchor=west] {Perturbation \((u,v)\)};
    \draw (1,2) node [anchor=west] {middle deck};
    \draw (2,2.5) node [anchor=west] {Inviscid approximation};
    \draw (2,1.85) node [anchor=west]
    {\(\tilde u(t,x,y) = A(t,x) U'(y)\)};
    \draw (2,1.4) node [anchor=west]
    {\(\tilde v(t,x,y)= - \pa_x A(t,x) U(y)\)};
    \draw (1,3.5) node [anchor=west] {upper deck};
    \draw (2,3.5) node [anchor=west] {Linearized Euler flow};
  \end{tikzpicture}
  \caption{Spatial scales of the triple deck model where
    \(\delta=\nu^{3/8}\) and \(\epsilon=\nu^{1/8}\).}
  \label{fig:spatial-scales}
\end{figure}
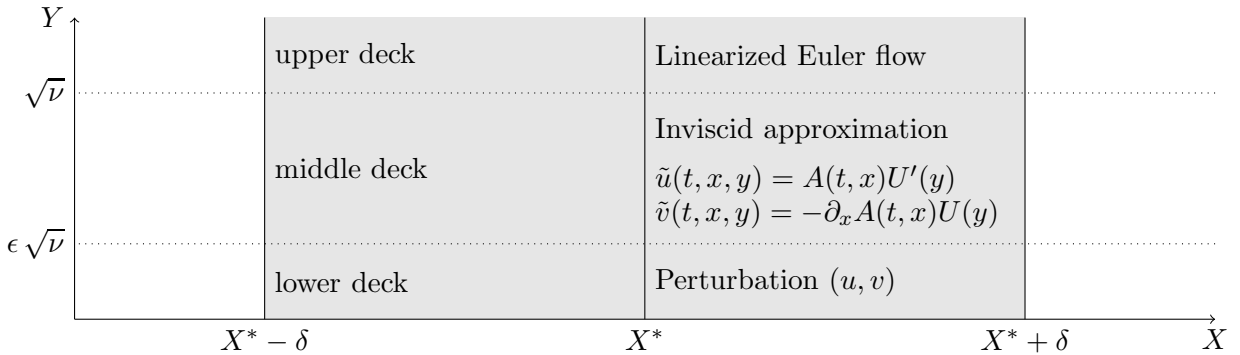

Although formulated in the 1960's, and extensively studied from a
numerical viewpoint since then, the triple deck system \eqref{TD} has
not been much investigated mathematically.  In the steady case, one
can mention the work \cite{MR2251455} of L. Planti\'e, focused on a
modification of the model: the displacement velocity $A= A(x)$ is
given while the pressure $p = p(x)$ is kept as an unknown.
Well-posedness is established under an assumption of non-decreasing
displacements, using Von Mises transform. In the unsteady case, the
only work we are aware of is the recent paper \cite{IyVi} by S. Iyer
and V. Vicol, which shows local in time well-posedness of \eqref{TD},
for data $u_0$ that are analytic in $x$, Sobolev in $y$, with further
gaussian decay in $y$. Let us stress that although analytic
well-posedness is well-known for the classical Prandtl equation, see
\cite{sammartino-calfisch-1998-boundary-layer-analytic,KuVi} extension
of such result to \eqref{TD} is uneasy. Indeed, the evolution equation
for $u$ in \eqref{TD} contains the annoying term
$\pa_x |\pa_x| A \approx \pa_x |\pa_x| u\vert_{z=\infty}$. This term
is not skew-symmetric in $L^2(\R \times \R_+)$, with potential severe
loss of two derivatives in $x$. To show a positive result in analytic
regularity, the authors have to combine two main ingredients. On one
hand, they control $A$ thanks to equation \eqref{consistency_eq}, in
which skew-symmetry of the Benjamin-Ono operator helps. On the other
hand, they control $u-A$ thanks to its rapid decay in $y$ and the use
of clever time-dependent cut-offs. We refer to \cite{IyVi} for all
necessary details.

\mspace Can we relax the assumption of analytic regularity for
well-posedness?  We remind that \eqref{P} is well-posed for any data
with Gevrey 2 regularity \cite{MR3925144}, and for Sobolev data that
are monotonic in $y$.  The triple deck model being supposedly a
refinement of the Prandtl one, it is natural to look for the same kind
of stability results. Encouragement can be found from the analysis of
the simplest linearization of \eqref{TD}, namely around $u(z) =
z$. The linearized system reads
\begin{equation} \label{eq:simple-linear-system}
  \begin{aligned}
    \pa_t u +  z \pa_x u + v - \pa^2_z u + \pa_x  |\pa_x| A & = 0, \\
    \pa_x u + \pa_z v & = 0, \\
    u\vert_{z=0} = v\vert_{z=0} & =  0. \\
    \lim_{z \rightarrow \infty} u   & = A.
  \end{aligned}
\end{equation}
Explicit calculations, sketched in Appendix \ref{appendixA}, can be
performed, and show that any family of eigenfunctions of the form
\begin{equation} \label{perturbation}
u_k(t,x,z) =  e^{\lambda_k t} e^{i kx} \hat{u}_k(z), \quad k \in \R
\end{equation}
satisfies $\Reel \lambda_k = O(1)$ as $k \rightarrow \pm\infty$, which is consistent with Sobolev well-posedness.

\mspace
Nevertheless, as we will show in this paper, there are monotonic shear flows $V_s(z)  = z + U_s(z)$ such that the linearization of \eqref{TD} around $u_s$ is ill-posed below analytic regularity. More precisely, we will prove that these linearized equations admit solutions of the form \eqref{perturbation}, with
$$\sigma_m := \liminf_{k \rightarrow +\infty} \Reel \lambda_k/k  > 0.$$
One could exhibit similarly solutions with
$\liminf_{k \rightarrow -\infty} \Reel \lambda_k/|k| > 0$.  This
prevents any general well-posedness statement for data $u_0$ that are
not analytic in $x$. In short, for any $T > 0$, one needs to impose a
bound of the form
$\|\hat{u}_0(k,\cdot)\| \le C e^{-(\sigma_m T) |k|}$, with some
appropriate norm $\| \cdot \|$ in variable $z$, to ensure a bound on the
solution over $(0,T)$. By the Paley-Wiener theorem, it is well-known
that such exponential decay in $|k|$ corresponds to analytic
regularity in $x$. In particular, one should not hope that the
analytic result of \cite{IyVi} can be improved in general. The next
section is dedicated to the statement of our main results.

\section{Results and strategy of proof}
We investigate in this paper linearizations of system \eqref{TD}, around shear flows of the form
$$u = z +  U_s(z), v =0.$$
Due to the diffusion term in \eqref{TD}, non-affine shear flows are
not solutions of the homogeneous triple deck equation. One way to
circumvent this issue is of course to start from the inhomogeneous
equation with source $-U''_s$. Another way to proceed is to consider
time-dependent shear flows, and argue that the time variation of the
flow is negligible at the time scale of the high frequency
instabilities that we shall discuss here. We refer to \cite{MR2601044}
for a rigorous reasoning in this second direction.

\mspace
We assume that
$$U_s(0) = 0, \quad  \lim_{z \rightarrow \infty} U_s = A_s  \in \R.$$
 We also assume for simplicity that $U_s$ is smooth, and that $U_s - A_s$ and their derivatives have fast decay at infinity. Less stringent assumptions could be extracted from the proof.
 The linearized system reads
\begin{equation} \label{LTD} \tag{LTD}
  \begin{aligned}
  \partial_t u + (U_s+z) \partial_x u - (1+U_s') \partial_x \int_0^z u + \partial_x |\partial_x| A - \pa^2_z u
  & =  0 \\
 u|_{z=0} & = 0 \\
  \lim_{z\to\infty} u &  = A.
  \end{aligned}
\end{equation}
We have expressed $v=-\int_0^y \pa_x u$ thanks to the divergence-free condition and the non-penetration condition $v\vert_{z=0} = 0$.
We are interested in the spectral analysis of \eqref{LTD} in the high
frequency regime, i.e.\ looking for eigenmodes  of the form
\begin{equation} \label{eigenmodes}
  u_k(t,x,y) = e^{-i k \mu_k  t} e^{ikx} \hat{u}_k(y), \quad
  A_k(t,x) =  \frac{1}{k} e^{-i k \mu_k \  t} e^{ikx}, \quad k \gg 1.
\end{equation}
Note that by linearity, we are allowed to fix  $\hat{A}_k = \frac{1}{k}$. We wish to exhibit a class of monotonic shear flows $U_s$ such that \eqref{LTD} has  non-trivial solutions of the form \eqref{eigenmodes} for $k$ large,  satisfying
 $\liminf_{k \rightarrow +\infty} \Im \mu_k  > 0$.
 To do so, we will follow the path introduced in \cite{MR3835243}  to analyze the stability properties of the linearized Interactive Boundary Layer model (IBL)
 \begin{equation} \label{LIBL}
 \begin{aligned}
 \partial_t u + (U_s+z) \partial_x u - (1+U_s') \partial_x \int_0^z u  - \pa^2_z u
  & =  \pa_t u_e + u_e \pa_x u_e  \\
   u_e - \sqrt{\nu} |\pa_x| \int_0^{+\infty} (u - u_e) dz & = 0 \\
 u|_{z=0} & = 0 \\
  \lim_{z\to\infty} u &  = u_e.
  \end{aligned}
\end{equation}
This path goes through the following steps:
\begin{enumerate}
\item We plug the formula \eqref{eigenmodes} in system \eqref{LTD},
  and reformulate our search for instability as a one-dimensional
  eigenvalue problem in variable $z$, with unknown eigenvalue $\mu_k$.
\item We take the formal limit $k \rightarrow +\infty$ of the
  eigenvalue problem, and we derive a necessary and sufficient
  condition on $U_s$ for the existence of an unstable eigenvalue
  $\mu_\infty$ to this limit eigenvalue problem.  To obtain such
  condition, we use tools from complex analysis, vaguely inspired by
  the work of O. Penrose on Vlasov-Poisson equilibria
  \cite{penrose1960}. First, we show that eigenvalues $\mu_\infty$ are
  the zeroes of an holomorphic function $\Phi_\infty$.  Namely,
  $\Phi_\infty(\mu) = \phi_{\mu, \infty}(0)$ for an explicit function
  $\phi_{\mu, \infty} = \phi_{\mu,\infty}(z)$. Then, we show that the
  existence of a zero $\mu_\infty$ in the unstable half plane
  $\{ \Im \mu > 0 \}$ amounts to a condition on the number of
  crossings of the positive real axis by some explicit curve related
  to $\Phi_\infty$. Examples (both numerical and analytical) of shear
  flows satisfying this condition are given.
 \item Eventually, we show that an instability at $k = +\infty$ persists at large but finite $k$. To do so, we express again the rescaled eigenvalue $\mu_k$ as the zero of a function
 $\Phi_k$. Again, $\Phi_k(\mu) =  \phi_{\mu,k}(0)$ for some function $\phi_{\mu,k} = \phi_{\mu,k}(z)$, but this function is no longer explicit. Roughly, it satisfies the resolvent equation of a Prandtl like operator. A keypoint of our analysis is to establish a stability estimate for this resolvent equation. Thanks to this estimate, we are then able to show that  $\Phi_k$ is holomorphic in $\{ \Im \mu > \delta \}$ for any $\delta > 0$ and $k$ large enough. We are also able to  show that for $k$ large and $\mu$ in  a compact set,  the solution
 $\phi_{\mu,k}$ is close to  $\phi_{\mu,\infty}$. We deduce from this that $\Phi_\infty$ and $\Phi_k$ are close in a neighborhood of $\mu_\infty$ for $k$ large enough, and conclude by Rouché's theorem.
\end{enumerate}
As a result of the analysis sketched above, we state our main theorem:
\begin{thm}[\textbf{Ill-posedness below analytic regularity}] \label{main_thm}
\phantom{x}\\
Assume $V_s(z) := z + U_s(z)$ has positive derivative on $\R_+$. Let
$ g(u) := \frac{V''_s}{(V'_s)^3}\vert_{u = V_s(y)}$. Assume that $g$
is strictly monotone in the neighborhood of each of its positive
zeroes, and define
\begin{equation*}
  n_\pm := \text{card} \, \Big\{  a > 0, \: g(a) = 0,  \:
  -\frac{1}{V'_s(0)} + a \textrm{PV} \int_0^\infty \frac{g(u)}{a-u} du
  > 0 , \: \pm g \: \text{ strictly increasing near $a$}  \Big\}.
\end{equation*}
Then, $n_\pm$ is finite, and if $n_+ - n_- \neq 0$, there exist
solutions of \eqref{LTD} of type \eqref{eigenmodes} with
$\liminf_{k \rightarrow \infty }\Im \mu_k > 0$.

\sspace
Moreover, there indeed exist shear flows $U_s$ such that $n_+ - n_- \neq  0$.
\end{thm}
The rest of the paper will be devoted to the proof of this
result. Section \ref{sec:k:infinite} is dedicated to the first two
steps alluded to above: rewriting of the problem as a 1-d eigenvalue
problem, and sharp analysis of the case $k=+\infty$.  Section
\ref{sec_resolvent_estimate} is devoted to the third step: we show
how to go from an instability at infinite $k$ to an instability at
finite $k$. Eventually, Section \ref{examples} collects examples,
either analytical or numerical, for which our instability criterion
applies. Let us stress that despite the similarities between
\eqref{LTD} and \eqref{LIBL}, a simple adaptation of the analysis
carried in \cite{MR3835243} is not enough to handle the triple deck
model. The boundary conditions at infinity, and notably the fact that
$u$ is unbounded far away, create specific difficulties. In
particular, we are unable to apply the kind of resolvent estimates
used in \cite{MR3835243} with such conditions at infinity. Instead, we
adapt the stability estimates that we used in \cite{MR3925144} to
obtain Gevrey 2 bounds for solutions of the classical Prandtl
equation.

\section{The infinite frequency spectral problem} \label{sec:k:infinite}
\subsection{Reduction}
We start by injecting solutions of type \eqref{eigenmodes} in \eqref{LTD}. From now on,  we shall work  in Fourier variables only, so we can use without confusion notation $u_k$ instead of $\hat{u}_k$. We find
\begin{align}
    \Big(-\mu_k + V_s \Big) u_k
    - V_s' \int_0^z u_k - \frac{1}{ik} \pa^2_z u_k
    & = - 1 ,\\
    u_k|_{z=0} & = 0,\\
    \lim_{z\to\infty} u_k & = \frac{1}{k},
\end{align}
where we remind that $V_s(z) = z + U_s(z)$. Like in \cite{MR3835243},
we further write $u_k$ in terms of a stream function \(\phi_k\) as
$u_k = (k^{-1}-\phi_k')$ yielding
\begin{equation} \label{OS_LTD}
  \begin{aligned}
    (\mu_k - V_s) \phi_k'  + V'_s \phi_k + \frac{1}{ik} \phi_k^{(3)}
    &  = -1 + \frac{1}{k} \big(\mu_k -  U_s +  z U_s'\big),\\
    \phi_k|_{z=0} = 0, \quad \phi'_k|_{z=0} & = \frac{1}{k}, \\
    \lim_{z\to\infty} \phi_k'(z) & = 0.
  \end{aligned}
\end{equation}
Note that  from the consistency equation, we have
\begin{equation} \label{OS_consistency}
\phi_k(\infty) = -1+ k^{-1} \big(\mu_k - A_s)
\end{equation}
which is again redundant to system \eqref{OS_LTD}. Let us note that the momentum equation is a third order ODE in $z$, and for general $\mu$ should require at most three boundary conditions for solvability. The fact that \eqref{OS_LTD}-\eqref{OS_consistency}  contains four boundary conditions is reminiscent of the fact that it is an eigenvalue problem, with unknowns $(\mu_k, \phi_k)$.

\mspace As explained in the previous section, in order to progress in
the analysis of solutions $(\mu_k, \phi_k)$ of \eqref{OS_LTD}, we
shall consider the formal limit of this system as
$k \rightarrow +\infty$. This raises a problem of boundary conditions:
indeed, at $k = +\infty$, the viscous term
$\frac{1}{ik} \phi_k^{(3)}$ disappears, and the operator in $z$
becomes first order.  Therefore, we drop the condition on
$\phi'_k(0)$, and consider the following infinite frequency spectral
problem
\begin{equation} \label{OS_LTD_inv}
   \begin{aligned}
(\mu_\infty - V_s) \phi_\infty' + V'_s \phi_\infty &  = -1,  \\
\phi_\infty(0) = 0, \quad   \lim_{z\to\infty} \phi_\infty'(z) & = 0. \\
  \end{aligned}
\end{equation}
The formal limit of the consistency condition \eqref{OS_consistency} is
\begin{equation} \label{OS_consistency_inv}
\phi_\infty(\infty) = -1.
\end{equation}
It is again redundant to \eqref{OS_LTD_inv}: as $\phi'_\infty$ still goes to zero at infinity,  taking the limit $z \rightarrow +\infty$ in the momentum equation  yields \eqref{OS_consistency_inv}.  This time, \eqref{OS_LTD_inv}-\eqref{OS_consistency_inv} contains three boundary conditions, for a first order system that would require {\it a priori} only one for solvability with an arbitrary given $\mu$.
We tackle a detailed analysis of this reduced eigenvalue problem in the next paragraph.

\subsection{Spectral analysis of the reduced eigenvalue problem}
We wish here to determine sharp conditions under which system \eqref{OS_LTD_inv} has a non-trivial solution $(\mu_\infty, \phi_\infty)$ with $\Im \mu_\infty > 0$.
Given $\mu \in \C\setminus\R_+$, we denote by $\phi = \phi_{\mu, \infty}$ the solution of
\begin{equation}  \label{OS_mu_inv}
   \begin{aligned}
(\mu - V_s) \phi' +  V'_s \phi&  = -1,  \\
   \lim_{z\to\infty} \phi'(z) & = 0. \\
  \end{aligned}
\end{equation}
As mentioned before,  for general $\mu$, as the first equation is
first order in $z$, one can only retain {\it a priori} one boundary
condition for solvability.  It is crucial that we retain here  the
condition on $\phi'$ at infinity, instead of the condition on $\phi$
at zero. This is a main difference with the treatment of the IBL model
in \cite{MR3835243}. Indeed, contrary to what happens in the IBL case,
the solution of the equation $(\mu- V_s) \phi' -  V'_s \phi  = 1$ with
$\phi(0) = 0$ is in general unbounded at infinity, due to the
unboundedness of $V_s$. Hence, it could not help to solve the
eigenvalue problem \eqref{OS_LTD_inv}. More generally no perturbative
argument could be based on such solution, that is associated to the
Dirichlet condition. On the contrary, system \eqref{OS_mu_inv} has an
explicit solution
\begin{equation*}
  \phi_{\mu, \infty}(z) =  (\mu - V_s(z)) \int_z^\infty
  \frac{1}{(\mu-V_s(y))^2} \, \dd y.
\end{equation*}
From this expression, one deduces that the consistency condition
\eqref{OS_consistency_inv} is satisfied, as expected.  Defining
\begin{equation} \label{def_Phi_infty}
  \Phi_\infty(\mu)  := \phi_{\mu, \infty}(0) =  \mu \int_0^\infty
  \frac{1}{(\mu-V_s(y))^2}\, \dd y,
\end{equation}
we see that $(\mu_\infty, \phi_\infty)$ will be a solution of
\eqref{OS_LTD_inv} if and only if
$$ \Phi_\infty(\mu_\infty) = 0, \quad \phi_\infty := \phi_{\mu_\infty, \infty}.$$

\mspace The rest of the paragraph is devoted to the proof of the
following proposition.
\begin{proposition} \label{main_prop}
Let  $V_s$, $g$ and $n_\pm$ as in Theorem \ref{main_thm}. Then, $\Phi_\infty$ has at least one zero in the unstable half-plane $\{\Im \mu > 0 \}$ if and only if the condition $n_+ - n_- \neq 0$  is satisfied.
\end{proposition}

\noindent
 We first state and prove  two lemmas describing the behaviour of $\Phi_\infty$ in various regions of $\C\setminus\R_+$:
\begin{lemma} \label{lemma_mu_far}
For any \(\delta>0\) there exists \(R > 0\) such that
  \begin{equation*}
    |\Phi_{\infty}(\mu) + 1| \le \delta
  \end{equation*}
  for all $\mu \in \C\setminus\R_+$ with \(|\mu| \ge R\).
\end{lemma}
\begin{proof} We write
\begin{equation} \label{re_expression_Phi_mu}
  \begin{aligned}
    \Phi_{\infty}(\mu) = \mu \int_0^\infty \frac{1}{V_s'}\;
    \underbrace{\frac{V_s'}{(\mu-V_s)^2}}_{=\left(\frac{1}{\mu-V_s}\right)'}
    \, \dd y
    &= -\frac{1}{V_s'(0)}  -\mu  \int_0^\infty \frac{V_s''}{V_s'^2}\, \frac{1}{V_s-\mu}\, \dd y
  \end{aligned}
\end{equation}
By elementary calculation, $\frac{-1}{V_s'(0)}  + \int_0^\infty \frac{V_s''}{V_s'^2}\, \dd y  = \frac{-1}{V_s'(0)} - \left[\frac{1}{V_s'}\right]_0^\infty  = -1$.
Hence we find
\begin{equation*}
  \Phi_{\infty}(\mu) + 1
  = \int_0^\infty \frac{V_s''}{V_s'^2}\;  \frac{V_s}{\mu-V_s}\, \dd y
  = \int_0^{\infty} \frac{g(u) u}{\mu- u} \, \dd u,
\end{equation*}
where $g$ is the function in Theorem \ref{main_thm}. Note that $g$ decays fast at infinity because $V''_s = U''_s$ does, and because $V'_s \rightarrow 1$ at infinity. We decompose:
$$ \int_0^{\infty} \frac{g(u) u}{\mu- u} du  =   \int_0^{+\infty} 1_{\{|u-\Re \mu| \ge 1\}} \frac{g(u) u}{\mu - u} du  +  \int_0^{+\infty} 1_{\{|u-\Re \mu| \le 1\}} \frac{g(u) u}{\mu - u} du   $$
By dominated convergence, the first term  goes to zero when $|\mu| \rightarrow +\infty$ in the region $\C\setminus\R_+$, and the second one goes to zero when $|\mu| \rightarrow +\infty$ in the region $\{\Re \mu \le -1\} \cup \{ \Im \mu > 1\}$. Eventually, for $ 0 < |\Im \mu| \le 1$, and  $\Re \mu \ge 1$, we write the second term as follows ($\mu = a+ib$):
$$ \int_0^{+\infty} 1_{\{|u-\Re \mu| \le 1\}} \frac{g(u) u}{\mu - u} du = \int_{-1}^{1} \frac{g(v+a)(v+a) - g(a) a}{ib-v} dv \: + \: g(a) a \int_{-1}^{1} \frac{1}{ib-v} dv $$
The first term vanishes when $|\mu| \rightarrow +\infty$, that is $a \rightarrow +\infty$, invoking again dominated convergence, and the second one goes to zero as well taking into account that $g(a) a \rightarrow 0$ and that $\lim_{b \rightarrow 0^+}  \int_{-1}^{1} \frac{1}{ib-v} dv$ exists by Plemelj formula. This concludes the proof of the lemma.
\end{proof}
\begin{lemma} \label{lemma_a_neg}
For all $\mu \in \C\setminus\R_+$ with $\Re \mu \le 0$,
  \begin{equation*}
    \Re \Phi_{\infty}(\mu) < 0.
  \end{equation*}
  Moreover,
  $$\lim_{\substack{\mu \rightarrow 0\\ \mu \in \C\setminus\R_+}} \Re \Phi_\infty(\mu) = -\frac{1}{V'_s(0)} < 0.$$
\end{lemma}
\begin{proof}
  From the definition \eqref{def_Phi_infty}, denoting $\mu = a + i b$,  we infer:
  \begin{equation*}
    \Re \Phi_\infty(\mu)
    = \int_0^\infty \frac{a ((a-V_s)^2-b^2) + 2b^2(a-V_s)}
    {|\mu-V_s(y)|^4}\, \dd y.
  \end{equation*}
  For the numerator we find
  \begin{equation*}
    a ((a-V_s)^2-b^2) + 2b^2(a-V_s)
    = a (a-V_s)^2 + b^2 (a-2V_s) \le 0
  \end{equation*}
  as $a \le 0$.  Unless \(a=b=0\) there exists also always \(z\) for
  which it is strictly negative, which concludes the proof of the
  first inequality. The other one is a simple consequence of the
  expression
  \begin{equation}  \label{expression_Phi_g}
    \Phi_{\infty}(\mu) =  -\frac{1}{V_s'(0)}  + \mu  \int_0^\infty  \frac{g(u)}{\mu-u}\, \dd y,
  \end{equation}
  see \eqref{re_expression_Phi_mu} and the definition of $g$ in
  Theorem~\ref{main_thm}. The integral
  $\int_0^\infty \frac{g(u)}{\mu-u}\, \dd u$ only diverges
  logarithmicaly as \(\mu \to 0\), hence the result.
\end{proof}

\noindent
We have now all the ingredients to conclude our analysis of the zeroes of $\Phi_\infty$ in $\{ \Im \mu > 0 \}$. By  Lemma \ref{lemma_mu_far}, there exists $R > 0$ such that
$$ |\Phi_\infty(\mu) +1| \le \frac{1}{4}, \quad \Im \mu > 0, \quad |\mu| \ge R.$$
Let
$$  \Omega_{\epsilon} := \{\mu \in \C,  \: \Im \mu > \epsilon,  \:  |\mu| \le R\} $$
With our choice of $R$, $\Phi_\infty$ has a zero in $\{ \Im \mu > 0 \}$ if and only if it has one in $\Omega_\eps$  for some $\eps > 0$ small enough.  As $\Phi_\infty$ is a holomorphic function, its zeroes in $\{ \Im \mu > 0 \}$ are isolated, so that we can restrict to $\eps$ along a sequence going to zero and such that $\Phi_\infty$ does not vanish at $\pa \Omega_\eps$.  Then, the number $n_\eps$ of its zeroes in $\Omega_\eps$, counted with multiplicity, is given by
\begin{equation*}
 n_\eps =  \frac{1}{2i \pi}\oint_{\partial\Omega_{\epsilon}} \frac{\Phi_\infty'(\zeta)}{\Phi_\infty(\zeta)}\, \dd \zeta.
\end{equation*}
Let \(\gamma_{\epsilon}\) be a direct parametrisation of the curve $\pa \Omega_\eps$. We have
\begin{equation*}
 \frac{1}{2i \pi}  \oint_{\partial\Omega_{\epsilon}}
  \frac{\Phi_\infty'(\zeta)}{\Phi_\infty(\zeta)}\, \dd \zeta
  =   \frac{1}{2i \pi} \int \frac{\Phi_{\infty}'(\gamma_\eps(t))}{\Phi_\infty(\gamma_\eps(t))}\gamma'_\eps(t)\, \dd t
  =  \frac{1}{2i \pi} \int \frac{(\Phi_{\infty} \circ \gamma_\eps)'(t)}{(\Phi_{\infty} \circ \gamma_\eps)(t)}\, \dd t = \frac{1}{2i \pi}  \int_{\Phi_\infty(\pa \Omega_\eps)} \frac{1}{\xi} \,  \dd \xi
\end{equation*}
so that the number of roots equals the winding number of the curve
$\Phi_\infty(\pa \Omega_\eps)$ around $0$. To compute this winding
number, one can choose a complex logarithm with a branch cut along the
positive real axis. The winding number is given by the sum of the
jumps of this logarithm, which corresponds to the number of crossings
of the curve $\Phi_\infty(\pa \Omega_\eps)$ with the positive real
axis. More precisely,
\begin{align} \label{diff_crossings}
  \frac{1}{2i \pi}  \oint_{\partial\Omega_{\epsilon}}
  \frac{\Phi_\infty'(\zeta)}{\Phi_\infty(\zeta)}\, \dd \zeta
  & =  \: \text{number of crossings from below}
    - \text{number of  crossings from above}.
\end{align}
We remind that the intersection of
$\Phi_\infty(\pa \Omega_\eps \cap \{|\mu| = R \})$ with the positive
real axis is empty.  Hence it remains to understand the crossings of
$\Phi_\infty([-R,R]+i\eps)$, for $\eps > 0$ going to zero. By
Lemma~\ref{lemma_a_neg}, there exists $\rho > 0$, such that
$\Re \Phi_\infty < 0$ for $-R < a < \rho$, so that we can restrict to
$a \in [\rho,R]$. Over this interval, formula \eqref{expression_Phi_g}
yields that
\begin{equation} \label{uniform_conv_Phi_infty}
  \begin{aligned}
    \Phi_\infty(a+ib) &=  -\frac{1}{V'_s(0)}
    + a \int_0^{+\infty} \frac{g(u)}{a+ib-u}\, \dd u +  i b
    \int_0^{+\infty} \frac{g(u)}{a+ib-u}\, \dd u \\
    & \xrightarrow[b \rightarrow 0^+]{} -\frac{1}{V'_s(0)}
    + a  PV \int_0^{+\infty} \frac{g(u)}{a-u}\, \dd u \: - \:   i \pi  a g(a)
  \end{aligned}
\end{equation}
using Plemelj formula, where the convergence is uniform in
$a \in [\rho, R]$.  We extend the definition of $\Phi_\infty$ over
$\R_+^*$ by
\begin{equation*}
  \Phi_\infty(a) :=   -\frac{1}{V'_s(0)} + a  PV \int_0^{+\infty}
  \frac{g(u)}{a-u}\, \dd u \: - \:   i \pi  a g(a), \quad a > 0.
\end{equation*}
We notice that the quantity $n_+$, resp.  $n_-$, defined in
Theorem~\ref{main_thm}, corresponds to the number of crossings from
above, resp. from below, of the curve $\Phi_\infty(\R_+^*)$ with the
positive real axis. Lemma~\ref{lemma_mu_far} being uniform in $b$, we
still have $\Phi_\infty(a) \rightarrow -1$ as $a \rightarrow
+\infty$. Also, as in Lemma~\ref{lemma_a_neg},
$\lim_{a \rightarrow 0^+} \Re \Phi_\infty(a) = -\frac{1}{V'_s(0)} <
0$.  It follows that $n_\pm$ is finite and coincides with the
crossings of $\Phi_\infty([\rho,R])$ up to taking $\rho$ smaller and
$R$ larger. Finally, from the uniform convergence in
\eqref{uniform_conv_Phi_infty}, we deduce that for $\eps > 0$ small
enough, \eqref{diff_crossings} is equal to $n_+ - n_-$. This concludes
the proof of Proposition~\ref{main_prop}.

\mspace

\section{Persistence of the instability at finite $k$} \label{sec_resolvent_estimate}
This section is devoted to the proof of Theorem \ref{main_thm},  except for the last statement, which will be considered in the next section. We assume that $n_+ - n_- \neq 0$, see the statement of the theorem. Our goal is to show that \eqref{LTD} has solutions of type \eqref{eigenmodes} with $\liminf_{k \rightarrow +\infty} \Im \mu_k > 0$. In other words, we need to prove that for all $k$ large enough, there exists $\mu_k, \phi_k$ solving \eqref{OS_LTD} with $\Im \mu_k \ge \delta > 0$ for  some $\delta$ independent of $k$.

\subsection{Resolvent Estimate}
The first step is to consider the following resolvent problem
\begin{equation} \label{resolventEQ}
  \left\{
  \begin{aligned}
  &(\mu-V_s) \psi' + (V_s') \psi + \frac{1}{ik} \psi'''
  = F \\
    &\psi'|_{z=0}=0, \lim_{z\to\infty} \psi'(z)=0, \lim_{z\to\infty} \psi = 0.
  \end{aligned}
  \right.
\end{equation}
Note that we do not prescribe the value of $\psi$ at zero.  Let
$\rho(z) = 1+z^m$, $m$ large. We denote by $\ipr{\,}{}$ the
$L^2(\rho)$ scalar product, and set
$\| \cdot \| = \ipr{\cdot}{\cdot}^{1/2}$.

\begin{proposition} \label{prop_resolvent} Let $\mu_m > 0$. There
  exists $k_m$, depending on $\mu_m$ such that for any $\mu$ with
  $\Im \mu \ge \mu_m$, for any $k \ge k_m$ and any $F \in L^2(\rho)$,
  the system \eqref{resolventEQ} has a unique solution
  $\psi = \psi_{\mu,k}$ of the form
  \begin{equation*}
    \psi =  \Big(\mu - V_s + \frac{1}{ik} \frac{\dd^2}{\dd z^2} \Big) A,
    \quad A'\vert_{z=0} = 0, \quad \lim_{z \to \infty} A = 0.
  \end{equation*}
  where $A \in H^3(\R_+)$ satisfies the estimates
  \begin{equation} \label{bound_A}
    \Im \mu\, \|A'\|^2 + \frac{1}{k} \|A''\|^2 \le \frac{4}{(\Im \mu)^3} \|F\|^2
  \end{equation}
  and
  \begin{equation} \label{bound_A_bis}
    \Im\mu\, \|A''\|^2_{L^2}  + \frac{1}{k}  \|A'''\|_{L^2}^2
    \le  C  \Big( \frac{\sqrt{k}}{(\Im \mu)^{9/2}}  + \frac{1}{(\Im \mu)^2} \Big)  \|F\|^2
  \end{equation}
  for a constant $C$ depending on $\mu_m$.
  Moreover, the map $\mu \rightarrow \psi_{\mu,k}(0)$ is analytic in $\{\Im \mu > \mu_m\}$.
\end{proposition}

\mspace
{\em Proof of the well-posedness statement in Proposition \ref{prop_resolvent}}.
We only detail the {\it a priori} estimates leading to the well-posedness. For the detailed construction of solutions in a similar context, see \cite{MR3835243}.  Inspired by our work \cite{MR3925144},  we introduce the solution $A$ of the system
\begin{equation} \label{eq_A_origin}
  \Big(\mu - V_s + \frac{1}{ik} \frac{\dd^2}{\dd z^2} \Big) A
  = \psi, \quad A'\vert_{z=0} = 0, \quad \lim_{z \to \infty} A = 0.
\end{equation}
Through differentiation in $z$, we get
\begin{equation*}
  \Big(\mu - V_s + \frac{1}{ik} \frac{\dd^2}{\dd z^2} \Big) A'  - V'_s A  =
  \psi'.
\end{equation*}
Inserting the two previous identities for $\psi$ and $\psi'$ in
Equation~\eqref{resolventEQ}, we find
$$ \Big(\mu - V_s + \frac{1}{ik} \frac{\dd^2}{\dd z^2} \Big)^2 A' -  \Bigl[  \Big(\mu - V_s + \frac{1}{ik} \frac{\dd^2}{\dd z^2} \Big) , V'_s\Bigl] A = F $$
that is
\begin{equation} \label{eq_A}
\Big(\mu - V_s + \frac{1}{ik} \frac{\dd^2}{\dd z^2} \Big)^2 A' -  \Big[ \frac{1}{ik} \frac{\dd^2}{\dd z^2} , U'_s  \Big] A = F
\end{equation}
 We introduce the solution $\varphi$ of
\begin{equation} \label{eq_phi}
\Bigl( \overline{\mu} - V_s  - \frac{1}{ik} \Bigl(\frac{d}{dz} +  \rho' \rho^{-1}\Bigr)^2 \Bigr) \varphi  = A', \quad \varphi\vert_{z=0} = 0, \quad  \lim_{z \to \infty} \varphi = 0.
\end{equation}
Taking the scalar product   of \eqref{eq_A} with $\varphi$,  we find
\begin{equation*}
  \ipr{\Big(\mu - V_s + \frac{1}{ik} \frac{\dd^2}{\dd z^2}  \Big)
    A'}{A'} - \frac{1}{ik} V'_s(0) A(0) \overline{\varphi'(0)} \rho(0)
  - \ipr{\Big[ \frac{1}{ik} \frac{\dd^2}{\dd z^2} , U'_s  \Big]
    A}{\varphi}  = \ipr{F}{\varphi}.
\end{equation*}
The boundary term at $z=0$  comes from  the integration by parts of  the diffusion term, taking into account that
$$\Big(\mu - V_s + \frac{1}{ik} \frac{\dd^2}{\dd z^2}\Big)A'(0) = \psi'(0) + V'_s(0) A(0) =  V'_s(0) A(0)  $$
We perform one more integration by parts and take the imaginary part to find
\begin{equation} \label{estim_A}
\begin{aligned}
 \Im \mu \|A'\|^2  + \frac{1}{k} \|A''\|^2  \: = \:  & \Im  \frac{1}{ik} \ipr{A''}{\rho' \rho^{-1} A'}
 - \Im  \frac{1}{ik}  V'_s(0)  A(0) \overline{\varphi'(0)} \rho(0)  \\
 & + \Im  \ipr{\Big[ \frac{1}{ik} \frac{\dd^2}{\dd z^2} , U'_s  \Big]  A}{\varphi} + \Im \ipr{F}{\varphi}.
 \end{aligned}
\end{equation}
 It remains to estimate the four terms at the right-hand side. Clearly,
\begin{equation} \label{estim_A_1}
  \Im  \frac{1}{ik}   \ipr{A''}{\rho' \rho^{-1} A'} \le \frac{1}{k}\|\rho' \rho^{-1}\|_\infty \|A''\| \, \| A'\| \le \frac{1}{2k} \|A''\|^2 + \frac{1}{2k} \|\rho' \rho^{-1}\|_\infty^2  \|A'\|^2.
\end{equation}
To bound the last terms, we first need to relate norms of $\varphi$ to
norms of $A$. We claim:
\begin{lemma} \label{lem_phi}
  For $k$ large enough (depending on $\mu_m$ and $\rho$), the solution
  $\varphi$ of \eqref{eq_phi} satisfies
  \begin{align*}
    & \|\varphi\|   \le \frac{\sqrt{2}}{\Im \mu} \|A'\|, \quad  \|\varphi\|_{L^2}  \le \frac{\sqrt{2}}{\Im \mu} \|A'\|_{L^2}, \\
    & \|\varphi'\|  \le \frac{\sqrt{k}}{\sqrt{\Im \mu}} \|A'\|, \quad \|\varphi'\|_{L^2}  \le \frac{\sqrt{k}}{\sqrt{\Im \mu}} \|A'\|_{L^2}, \quad  \|\varphi''\|_{L^2}  \le C k \|A'\|_{L^2},
  \end{align*}
  where $C$ depends also on $\mu_m$.
\end{lemma}

\mspace
We take the $L^2(\rho)$ scalar product of \eqref{eq_phi}  with $\varphi$, and retain the imaginary part:
\begin{align*}
\Im \mu \|\varphi\|^2 + \frac{1}{k} \|\varphi'\|^2 & = \Im \frac{1}{ik} \ipr{\rho' \rho^{-1} \varphi}{\varphi'}   - \Im \ipr{A'}{\varphi} \\
& \le \frac{1}{2k}\|\rho' \rho^{-1}\|^2_\infty \|\varphi\|^2  +   \frac{1}{2k} \|\varphi'\|^2 + \frac{1}{2 \Im \mu}\|A'\|^2 + \frac{\Im \mu}{2} \|\varphi\|^2
\end{align*}
For $ \frac{1}{2k}\|\rho' \rho^{-1}\|^2_\infty \le \frac{\mu_m}{4}$, we get in particular
$$ \frac{\Im \mu}{4} \|\varphi\|^2 \le \frac{1}{2 \Im \mu}\|A'\|^2 $$
which implies the first estimate. We also find
$$ \frac{1}{2k} \|\varphi'\|^2 \le   \frac{1}{2 \Im \mu}\|A'\|^2  $$
which implies the second estimate. Similar (and even simpler) calculations yield
$$ \|\varphi\|_{L^2}   \le \frac{\sqrt{2}}{\Im \mu} \|A'\|_{L^2}, \quad  \|\varphi'\|_{L^2}  \le \frac{\sqrt{k}}{\sqrt{\Im \mu}} \|A'\|_{L^2}. $$
 To obtain the last inequality of the lemma, we multiply by  $\varphi''$ and integrate. Denoting $\ip{\,}{}$ the classical $L^2$ scalar product, we get
\begin{align*}
  \frac{1}{k} \|\varphi''\|^2_{L^2} + \Im \mu \| \varphi'\|_{L^2}^2
  & = \Im  \frac{1}{ik}  \ip{\big( 2  \big(\rho' \rho^{-1}\big)' \varphi' + \big(\rho' \rho^{-1}\big)''  \varphi}{\varphi''}   + \Im \ip{V'_s  \varphi}{\varphi'} + \Im \ip{A'}{\varphi''} \\
  & \le  \frac{1}{4k} \|\varphi''\|_{L^2}^2  +  \frac{1}{k} \| 2 \big(\rho' \rho^{-1}\big)' \|^2_{\infty}  \|\varphi'\|_{L^2}^2 +   \frac{1}{4k} \|\varphi''\|_{L^2}^2  +  \frac{1}{k}   \|  \big(\rho' \rho^{-1}\big)'' \|^2_{\infty}   \|\varphi\|_{L^2}^2 \\
  &\quad +   \frac{\Im \mu}{4}  \|\varphi'\|_{L^2}^2  +  \frac{1}{\Im \mu} \|V'_s\|_{\infty}^2  \|\varphi\|_{L^2}^2  \\
  &\quad +  \frac{1}{4k} \|\varphi''\|_{L^2}^2 +  k \|A'\|_{L^2}^2.
\end{align*}
Using the previous estimate for $\|\varphi\|_{L^2}$, we deduce that for $k$ large enough
\begin{align*}
 \frac{1}{4k} \|\varphi''\|_{L^2}^2  + \frac{\Im \mu}{4} \| \varphi'\|_{L^2}^2   \le \big(C + k \big)  \|A'\|_{L^2}^2
\end{align*}
for a constant $C$ depending on $\mu_m$. The last estimate of the lemma follows.

\mspace
We now go back to the identity \eqref{estim_A}, where we have to bound the last three terms at the right-hand side. The easiest is
\begin{equation} \label{estim_A_2}
 \Im \ipr{F}{\varphi}  \le \frac{1}{(\Im\mu)^3} \|F\|^2  + \frac{(\Im \mu)^3}{4} \|\varphi\|^2 \le \frac{1}{(\Im\mu)^3} \|F\|^2 + \frac{\Im \mu}{2} \|A'\|^2.
 \end{equation}
The commutator term splits into
\begin{align*}
  \Im  \ipr{\Big[ \frac{1}{ik} \frac{\dd^2}{\dd z^2} , U'_s  \Big]  A}{\varphi} & =  \Im   \ipr{\frac{2}{ik} U''_s A'}{\varphi}  + \Im   \ipr{\frac{1}{ik} U'''_s A}{\varphi} \\
                                                                                & \le \frac{2}{k} \|U''_s\|_\infty \|A'\| \, \|\varphi\|  +   \frac{C_H}{k} \|U'''_s(1+z)\|_\infty \|A'\| \, \|\varphi\|,
\end{align*}
where we used the Hardy inequality $\|\frac{A}{1+y}\| \le C_H
\|A'\|$. Using the first estimate in the lemma, we end up with
\begin{equation} \label{estim_A_3}
 \Im  \ipr{\Big[ \frac{1}{ik} \frac{\dd^2}{\dd z^2} , U'_s  \Big]  A}{\varphi}  \le \frac{C}{(\Im \mu) k} \|A'\|^2.
\end{equation}
Eventually, for the trace term, we have through Sobolev imbedding:
\begin{equation} \label{estim_A_4}
  \begin{aligned}
    \Im  \frac{1}{ik} (U_s'(0) + 1) A(0) \overline{\varphi'(0)} \rho(0) & \le \frac{C}{k} \|A\|_\infty \|\varphi'\|^{1/2}_{L^2} \|\varphi''\|_{L^2}^{1/2} \\
    & \le \frac{C}{k} \|\rho^{-1/2}\|_{L^2}  \|A'\| \, \|\varphi'\|^{1/2}_{L^2} \,  \|\varphi''\|_{L^2}^{1/2} \\
    & \le \frac{C'}{k^{1/4} (\Im \mu)^{1/4}}   \|A'\| \, \|A'\|_{L^2}
  \end{aligned}
\end{equation}
for constants $C,C'$ depending on $\mu_m$.

\mspace
Collecting \eqref{estim_A_1}-\eqref{estim_A_2}-\eqref{estim_A_3}, we find that for $k$ large enough, depending on $\mu_m$,
\begin{equation*}
\Im \mu \|A'\|^2 + \frac{1}{k} \|A''\|^2 \le \frac{4}{(\Im \mu)^3} \|F\|^2
\end{equation*}
that is \eqref{bound_A}.

\mspace
To establish estimate \eqref{bound_A_bis}, we introduce the solution $\varphi_1$ of
\begin{equation} \label{eq_phi_1}
\Bigl( \overline{\mu} - V_s - \frac{1}{ik} \frac{\dd^2}{\dd z^2} \Big)\varphi_1  = A''', \quad \varphi_1\vert_{y=0} = 0, \quad  \lim_{y \to \infty} \varphi_1 = 0.
\end{equation}
Taking the usual  $L^2$ scalar product of \eqref{eq_A} with $\varphi_1$, we obtain
\begin{equation*}
  \ip{\Big(\mu - V_s+ \frac{1}{ik} \frac{\dd^2}{\dd z^2} \Big) A'}{A'''}
  + \frac{1}{ik} V_s'(0) A(0) \overline{\varphi'_1(0)}   -
  \ip{\Big[ \frac{1}{ik} \frac{\dd^2}{\dd z^2} , U'_s  \Big] A}{\varphi_1}
  = \ip{F}{\varphi_1}.
\end{equation*}
Hence,
\begin{align*}
\Im\mu \|A''\|_{L^2}^2  + \frac{1}{k}  \|A'''\|_{L^2}^2 & = \Im \ip{V'_s  A'}{A''}  \\
& + \Im  \frac{1}{ik} V_s'(0)  A(0) \overline{\varphi'_1(0)} \rho(0) \\
&  +\Im  \ip{\Big[ \frac{1}{ik} \frac{\dd^2}{\dd z^2} , U'_s  \Big]  A}{\varphi_1} - \Im  \ip{F}{\varphi_1}
\end{align*}
The last three terms can be treated like before, replacing $\varphi$ by $\varphi_1$. First, the estimates in Lemma  \ref{lem_phi} are replaced by
\begin{align*}
&  \|\varphi_1\|_{L^2}   \le \frac{\sqrt{2}}{\Im \mu} \|A'''\|_{L^2}, \quad   \|\varphi_1'\|_{L^2}  \le \frac{\sqrt{k}}{\sqrt{\Im \mu}} \|A'''\|_{L^2}, \quad \|\varphi_1''\|_{L^2}  \le C k \|A'''\|_{L^2},
\end{align*}
with the same proof. Then, estimate \eqref{estim_A_3} becomes
\begin{equation*}
 \Im  \ip{\Big[ \frac{1}{ik} \frac{\dd^2}{\dd z^2} , U'_s  \Big] A}{\varphi_1}
 \le \frac{C}{(\Im \mu) k} \|A'\| \, \|A'''\|_{L^2}
 \le \frac{1}{4k}  \|A'''\|_{L^2}^2  + \frac{C^2}{(\Im \mu)^2 k} \|A'\|^2.
\end{equation*}
Proceeding as in \eqref{estim_A_4}, we find
\begin{equation*}
  \Im  \frac{1}{ik} V_s'(0) A(0)
  \overline{\varphi'_1(0)}  \le  \frac{C}{k^{1/4}(\Im
    \mu)^{1/4}}\|A'\|  \, \|A'''\|_{L^2} \le \frac{1}{4k}
  \|A'''\|_{L^2}^2 +  \frac{C^2 \sqrt{k}}{\sqrt{\Im \mu}} \|A'\|^2.
\end{equation*}
Also,
$$ - \Im  \ip{F}{\varphi_1} \le \|F\|_{L^2} \, \|\varphi_1\|_{L^2} \le  \|F\|_{L^2} \frac{\sqrt{2}}{\Im \mu} \|A'''\|_{L^2} \le \frac{1}{4k} \|A'''\|^2_{L^2} + \frac{2}{(\Im \mu)^2} \|F\|_{L^2}^2 $$
The remaining term is controlled by
$$  \Im \ip{V'_s  A'}{A''} \le C \|A'\|_{L^2} \, \|A''\|_{L^2} \le  \frac{C^2}{2 (\Im \mu)} \|A'\|^2 \, + \frac{\Im \mu}{2} \|A''\|_{L^2} $$
Collecting these bounds, we end up with
\begin{equation*}
  \frac{\Im\mu}{2} \|A''\|_{L^2}^2  + \frac{1}{4k}  \|A'''\|_{L^2}^2
  \le  C \Big(  \frac{\sqrt{k}}{\sqrt{\Im \mu}} \|A'\|^2
  +  \frac{1}{\Im \mu}\|A'\|^2   + \frac{2}{(\Im \mu)^2} \|F\|^2_{L^2} \Big)
\end{equation*}
From this bound and  \eqref{bound_A}, we deduce \eqref{bound_A_bis}
\begin{equation*}
  \Im\mu \|A''\|^2_{L^2}  + \frac{1}{k}  \|A'''\|_{L^2}^2
  \le  C  \Big( \frac{\sqrt{k}}{(\Im \mu)^{9/2}}  + \frac{1}{(\Im \mu)^2} \Big)  \|F\|^2
\end{equation*}
for some constant $C$ depending on $\mu_m$.

 \mspace
 {\em Proof of the analyticity statement in Proposition \ref{prop_resolvent}}.
 The last thing to be shown is the analyticity of the map $\mu \rightarrow \psi_{\mu,k}(0)$. As before, we write $\psi = \psi_{\mu,k}$. We proceed as follows: let $\chi_n(z) = \chi\big(\frac{z}{n}\big)$, for some smooth non-negative $\chi$ which is one near the origin and zero in the large. We consider the approximate problem
\begin{equation} \label{resolventEQn}
  \left\{
  \begin{aligned}
  &(\mu-(U_s+ z \chi_n)) \psi' + (U_s' + (z \chi_n)') \psi + \frac{1}{ik} \psi'''
  = F \\
    &\psi'|_{z=0}=0, \lim_{z\to\infty} \psi'(z)=0, \lim_{z\to\infty} \psi = 0.
  \end{aligned}
  \right.
\end{equation}
Above calculations apply to to this approximate system as well, and yield a solution
$$\psi_n  =  \Big(\mu - (U_s + z \chi_n) + \frac{1}{ik} \frac{\dd^2}{\dd z^2} \Big) A_n,   \quad A_n'\vert_{z=0} = 0, \quad \lim_{z \to \infty} A_n = 0 $$
 where $A_n$ satisfies the same estimates \eqref{bound_A}-\eqref{bound_A_bis} uniformly in $n$. In particular, $A'_n \in L^2(\rho)$. The difference with the original  system  is that this implies  $\psi'_n \in L^2(\rho)$:  it can be deduced from  the formula
  $$ \psi'_n =  \Big(\mu - (U_s + z \chi_n) + \frac{1}{ik} \frac{\dd^2}{\dd z^2} \Big) A'_n  - (U'_s+(z \chi_n)') A_n $$
 as the base flow  $U_s + z \chi_n$ is not diverging at infinity. Using the bounds for $A_n$, one obtains an estimate of the type
 $\|\psi'_n\| \le  C_n \|F\|$, with a bound that may depend on $n$ through $\chi_n$, but is uniform in $\mu$ inside  $\{ \Im \mu > \mu_m\}$. Using further the equation in \eqref{resolventEQn}, we find that $\psi'_n \in H^2(\rho)$, and a bound of the form $\|\psi'_n\|_{H^2(\rho)} \le C_n \|F\|$. Introducing the operator
 \begin{equation*}
   \mathcal{L}_n : H^2(\rho) \cap H^1_0(\rho)
  \to L^2(\rho), \quad u \mapsto - (U_s+ z \chi_n) u -
  (U_s' + (z \chi_n)') \int_{z}^{\infty} u + \frac{1}{ik}  u''
\end{equation*}
we then know that its resolvent $(\mu + \mathcal{L}_n)^{-1}$  is well-defined for $\Im \mu >  \mu_m$, and as any resolvent operator is analytic in $\mu$. Hence,
$\psi'_n = (\mu + \mathcal{L}_n)^{-1} F$ is analytic in $\mu$ with values in $H^2(\rho)$, and by imbedding $\mu \rightarrow \psi_n(0)$ is analytic as well.

\mspace
To conclude that $\mu \rightarrow \psi(0)$ is analytic, it remains to show that $\psi_n(0) \xrightarrow[n \rightarrow +\infty]{} \psi(0)$ uniformly on the compact sets of $\{ \Im \mu > \mu_m\}$.  Note that, by \eqref{eq_A_origin}, we have
\begin{align*}
  |\psi_n(0) - \psi(0)|
  & = |\mu (A_n(0) - A(0))  + \frac{1}{ik} \pa^2_z (A_n - A)(0)| \\
  & \le C \Big( \|A'_n - A'\|_{L^2(\tilde{\rho})} + \|A''_n - A''\|_{L^2(\tilde \rho)} + \|A_n''' - A'''\|_{L^2}\Big)
\end{align*}
for all $\mu$ in a compact set $K$, for any weight function $\tilde \rho \ge 1$ such that $\frac{1}{\tilde \rho} \in L^1(\R_+)$, where the  constant $C$ depends on $K$ and $k$. The last step is to establish that the right-hand side goes to zero (uniformly in $\mu \in K$), which can be done through an estimate of the difference.
Namely, combining \eqref{eq_A} and its analogue
$$ \Big(\mu - (U_s + z\chi_n ) + \frac{1}{ik} \frac{\dd^2}{\dd z^2} \Big)^2 A_n' -  \Big[ \frac{1}{ik} \frac{\dd^2}{\dd z^2} , U'_s + (z \chi_n)'  \Big] A_n = F $$
we see that
$$ \Big(\mu - V_s + \frac{1}{ik} \frac{\dd^2}{\dd z^2} \Big)^2 (A_n - A)' -  \Big[ \frac{1}{ik} \frac{\dd^2}{\dd z^2} , U'_s  \Big] (A_n - A) = R_n$$
where, denoting $\psi_n = z (\chi_n -1)$:
\begin{align*}
R_n & = \Big[ \frac{1}{ik} \frac{\dd^2}{\dd z^2} ,  \psi_n' \Big] A'_n \\
& -  \Big(\mu - V_s + \frac{1}{ik} \frac{\dd^2}{\dd z^2} \Big) (\psi_n A_n)  -  \psi_n  \Big(\mu - V_s + \frac{1}{ik} \frac{\dd^2}{\dd z^2} \Big)  A_n \\
& - \psi_n^2 A_n.
\end{align*}
Applying \eqref{bound_A} and \eqref{bound_A_bis} with weight
$\tilde \rho = (1+z)^{-8} \rho$ instead of $\rho$, we find that
\begin{equation*}
  \|A'_n - A'\|_{L^2(\tilde{\rho})}
  + \|A''_n - A''\|_{L^2(\tilde \rho)} + \|A_n''' - A'''\|_{L^2} \le C
  \|R_n\|_{L^2(\tilde \rho)}
\end{equation*}
for a constant $C$ that again may depend on $K$ or $k$. Eventually, using that $A'_n$ is bounded uniformly in $n$ in $H^1(\rho)$, one can check that $\|R_n\|_{L^2(\tilde \rho)}$ goes to zero as $n \to \infty$. For instance,
\begin{align*}
  \left\|\Big[ \frac{1}{ik} \frac{\dd^2}{\dd z^2} ,  \psi_n' \Big]
  A'_n\right\|_{L^2(\tilde \rho)}
  &\lesssim \|\psi'''_n (1+z)^{-2}\|_\infty \|A'_n\|_{L^2(\rho)} +   \|\psi''_n (1+z)^{-2}\|_\infty
 \|A''_n\|_{L^2(\rho)} \\
  &\lesssim \|\psi'''_n (1+z)^{-4}\|_\infty
    +  \|\psi''_n (1+z)^{-4}\|_\infty
    \lesssim  \frac{1}{n}.
\end{align*}
All other terms defining $R_n$ can be treated with similar ideas.

\subsection{Conclusion}
We remind that our goal is to find solutions $(\mu_k, \phi_k)$ of
\eqref{OS_LTD} with $\liminf_{k \rightarrow +\infty} \mu_k > 0$.  From
the analysis of Section~\ref{sec:k:infinite}, notably
Proposition~\ref{main_prop}, we already know that the system
\eqref{OS_LTD_inv} has a solution $(\mu_\infty, \phi_\infty)$ with
$\Im \mu_\infty > 0$. We shall find $\mu_k, \phi_k$ with $\mu_k$ close
to $\mu_\infty$. We fix $\mu_m = \frac{1}{2} \Im \mu_\infty$.

\mspace We shall first prove that for any $\mu$ with
$\Im \mu \ge \mu_m$ and $k \ge k_m$, with $k_m$ given by Proposition
\ref{prop_resolvent}, one can construct a solution
$\phi = \phi_{\mu,k}$ of
\begin{equation} \label{OS_mu}
  \begin{aligned}
    (\mu - V_s) \phi'  + V'_s \phi + \frac{1}{ik} \phi^{(3)}
    &  = -1 + \frac{1}{k} \big(\mu -  U_s +  z U_s'\big),\\
    \phi'|_{z=0} & = \frac{1}{k},  \quad  \lim_{z\to\infty} \phi'(z)  = 0.
  \end{aligned}
\end{equation}
The first equation implies the compatibility condition $\phi(\infty) = -1 +  \frac{1}{k}  \big(\mu -  A_s)$.  We  look for an approximation of $\phi$  under the form
 $$ \phi_{app} = \phi_{\mu,\infty} + \phi_{bl} $$
 where $\phi_{\mu,\infty}$ solves \eqref{OS_mu_inv} and where
 $\phi_{bl} = \phi_{bl}(z)$ is a boundary layer term that allows to
 recover the right Neumann boundary condition at $z=0$ and the right
 Dirichlet condition at infinity.  Namely, $\phi_{bl}$ is the solution
 of
\begin{equation*}
 \mu \phi'_{bl} + \frac{1}{ik} \phi_{bl}^{(3)} = 0,  \quad
 \phi_{bl}'(0)  = - \phi'_{\mu,\infty}(0) + \frac{1}{k},    \quad
 \phi'_{bl}(\infty)  = 0, \quad  \phi_{bl}(\infty) =  \frac{1}{k}
 \big(\mu -  A_s).
\end{equation*}
The solution is explicitly given by
\begin{equation} \label{def_phi_bl}
  \phi_{bl}(z) = \frac{1}{k}  \big(\mu -  A_s \big)
  + \int_{\infty}^z e^{-\sqrt{-ik\mu}\, y} \dd y
  \, \Big( - \phi_{\mu,\infty}'(0) + \frac{1}{k} \Big).
\end{equation}
Thanks to this choice, it is straightforward to check that the
difference $\psi = \phi - \phi_{app}$ satisfies a system of type
\eqref{resolventEQ} with
\begin{equation} \label{def:F}
 F :=  V_s \phi_{bl}' - V'_s \phi_{bl} - \frac{1}{ik} \phi^{(3)}_{\mu,\infty}  + \frac{1}{k} \big(\mu -  U_s +  z U_s'\big).
\end{equation}
By Proposition~\ref{prop_resolvent}, there exists a solution $\psi$ to
this system, so that
\begin{equation*}
  \phi_{\mu,k} = \phi_{app} + \psi = \phi_{\mu,\infty} + \phi_{bl} + \psi
\end{equation*}
defines  a solution of \eqref{OS_mu}.   Defining
\begin{equation*}
  \Phi_{k}(\mu) := \phi_{k,\mu}(0).
\end{equation*}
We shall prove that for $k$ large enough, there exists $\mu_k$ in the
disk
\(D(\mu_\infty, \frac{1}{2} \Im \mu_\infty) :=
\{ {z \in \C}:
|\mu_\infty-z| < \frac{1}{2} \Im \mu_\infty\}\) such that
$\Phi_{k}(\mu_k) = 0$. Hence, $(\mu_k, \phi_k := \phi_{\mu_k,k})$ will
be the desired solution to \eqref{OS_LTD}. More precisely, we state
\begin{proposition} \label{prop_Rouche}
  There exist constants $C,K > 0$ such that \(\Phi_k\) is holomorphic
  in \(D(\mu_\infty, \frac{1}{2} \Im \mu_\infty)\) and
  \begin{equation*}
    |\Phi_{k}(\mu) - \Phi_\infty(\mu)| = |\phi_{\mu,k}(0) -
    \phi_{\mu,\infty}(0)| \le C k^{-1/4},
    \qquad
    \forall \mu \in D(\mu_\infty, \frac{1}{2} \Im \mu_\infty), k \ge K.
  \end{equation*}
\end{proposition}

\mspace Before proving this proposition, let us show how it implies
the existence of a zero $\mu_k$ of $\Phi_k$. We already know that
$\Phi_\infty$ is holomorphic, and that $\mu_\infty$ is one of its
zeroes, therefore isolated. Hence, for $\delta > 0$ small enough,
$\eps := \inf_{|\mu - \mu_\infty|=\delta} |\Phi_\infty(\mu)| > 0$. For
all $k$ large enough so that $|C k^{-1/4}| \le \frac{\eps}{2}$, we
conclude by Proposition~\ref{prop_Rouche} and Rouché's theorem that
$\Phi_k$ has a zero in $D(\mu_\infty, \delta)$.

\begin{proof}[Proof of Proposition~\ref{prop_Rouche}]
  We have
  $$ \Phi_{k}(\mu) - \Phi_\infty(\mu) = \phi_{\mu,k}(0) - \phi_{\mu,\infty}(0) = \phi_{bl}(0) + \psi(0). $$
  First, from \eqref{def_phi_bl},
  \begin{equation} \label{phi_bl_zero}
    \phi_{bl}(0) =  \frac{1}{k}  \big(\mu -  A_s\big) -  \frac{1}{\sqrt{-ik\mu}} \, \Big( - \phi_{\mu,\infty}(0) + \frac{1}{k} \Big)
  \end{equation}
  so that for all $\mu \in D(\mu_\infty, \delta)$ it holds that
  $|\phi_{bl}(0)| \le C k^{-1/2}$. Then, with the notations of
  Proposition~\ref{prop_resolvent},
  \begin{align*}
    |\psi(0)| &= \Big|\Big( \mu - V_s + \frac{1}{ik} \frac{\dd^2}{\dd z^2}
                \Big) A(0)\Big|
                = \Big| \mu A(0)  + \frac{1}{ik}  A''(0)\Big| \\
              & \le C\left( \|A'\| + \|A''\|^{1/2}_{L^2} \|A^{(3)}\|^{1/2}_{L^2}  \right) \\
              & \le C k^{1/2} \|F\|
  \end{align*}
  where the last inequality is a consequence of the estimates in
  Proposition~\ref{prop_resolvent}, with $F$ defined in
  \eqref{def:F}. Note that
  \begin{align*}
    |F(z)| & \le \Big\|\frac{V_s}{z}\Big\|_{L^\infty}  |z \phi'_{bl}(z)| +  \| V'_s \|_{L^\infty}  \big|\phi_{bl}(z) - \frac{1}{k}(\mu - A_s)\big| +  |V'_s(z) - 1| \frac{1}{k} (\mu - A_s) \\
           &\quad + \frac{1}{k}  |U_s(z) - A_s | + \frac{1}{k}  |z U'_s(z)|.
  \end{align*}
  Using this inequality, one gets $\|F\| \le C k^{-3/4}$, and eventually
  $|\psi(0)| \le C k^{-1/4}$. The estimate of
  Proposition~\ref{prop_Rouche} follows. As regards the analyticity of
  $\mu \mapsto \Phi_k$, it follows from the analyticity of
  $\mu \mapsto \phi_{bl}(0)$ and of $\mu \rightarrow \psi(0)$. The
  former is deduced directly from formula~\eqref{phi_bl_zero}, having in
  mind the analyticity of $\phi_{\mu,\infty}(0) = \Phi_\infty(0)$. The
  latter is deduced from the analyticity statement of
  Proposition~\ref{prop_resolvent}. More precisely, the statement is
  given there for a source term $F$ that is independent of $\mu$, but it
  is still true, with the same proof, for an $F$ analytic in $\mu$,
  which is the case here, {\it cf.} \eqref{def:F}. This concludes the
  proof of the proposition.
\end{proof}

\section{Examples of instabilities} \label{examples}

In this final part of the paper, we exhibit examples, both numerical
and analytical, for which the quantity $n_+ - n_-$ mentioned in
Theorem~\ref{main_thm} is indeed non-zero.

For the theoretical existence of unstable modes, we note that by
\eqref{def_Phi_infty} can be written as
\begin{equation*}
  \Phi_\infty(\mu) = \int_0^\infty q_\mu(V_s(y))\, \dd y
\end{equation*}
with the function \(q_\mu : u \mapsto \mu (\mu-u)^2\). For a fixed
\(\mu\), we now look at \(q_\mu([0,\infty))\) and by the idea of
Section~6 of \cite{MR3835243}, we can construct smooth profiles
\(V_s\) with \(\Phi_\infty(\mu) = 0\) if and only if the origin is in
the interior of the convex hull of \(q_\mu([0,\infty))\). Indeed for
\(\Re \mu > 0\) and \(\Im \mu > 0\), the origin is in the convex hull
of \(q_\mu([0,\infty))\) so that unstable modes exist, see
Figure~\ref{fig:image-q-mu} as an example.

\begin{figure}[htb]
  \centering
  \begin{tikzpicture}
    \begin{axis}[
      xlabel={\(\Re q_\mu\)},
      ylabel={\(\Im q_\mu\)},
      ]
      \addplot [very thick]
      table {%
        0.5 -0.5
        0.499965914167813 -0.506745190379735
        0.499862432248654 -0.513557927222098
        0.499687694802296 -0.520438163323987
        0.499439812020407 -0.527385818535894
        0.499116863803816 -0.534400778464592
        0.498716899873421 -0.541482893147281
        0.498237939916449 -0.548631975697368
        0.497677973769829 -0.555847800922129
        0.49703496164246 -0.563130103912561
        0.496306834378236 -0.570478578605797
        0.495491493761698 -0.577892876320535
        0.494586812868248 -0.585372604266003
        0.493590636460895 -0.59291732402506
        0.492500781435541 -0.600526550012123
        0.491315037316841 -0.608199747906691
        0.490031166806738 -0.615936333063329
        0.488646906387764 -0.623735668899083
        0.487159966983266 -0.63159706525938
        0.485568034676723 -0.639519776763586
        0.483868771492347 -0.647503001131508
        0.482059816239192 -0.655545877492224
        0.48013878542099 -0.663647484676749
        0.47810327421398 -0.671806839496203
        0.475950857514969 -0.680022895007206
        0.473679091061893 -0.688294538766438
        0.471285512629125 -0.696620591076367
        0.468767643299797 -0.704999803224354
        0.466122988817355 -0.713430855717433
        0.46334904101858 -0.721912356515267
        0.460443279350244 -0.730442839263887
        0.457403172471594 -0.739020761533028
        0.454226179944741 -0.747644503059991
        0.450909754015051 -0.756312364003171
        0.447451341483522 -0.765022563208523
        0.443848385673117 -0.773773236492436
        0.440098328490886 -0.782562434944647
        0.436198612587679 -0.79138812325501
        0.432146683617118 -0.800248178068116
        0.427939992595405 -0.809140386369926
        0.423575998363419 -0.818062443910796
        0.419052170152421 -0.827011953669402
        0.414365990254543 -0.835986424362317
        0.409514956799096 -0.844983269004118
        0.404496586635524 -0.853999803523126
        0.399308418323699 -0.863033245438026
        0.393948015232011 -0.872080712600817
        0.388412968743522 -0.88113922201169
        0.382700901570219 -0.890205688711629
        0.376809471175154 -0.89927692475866
        0.370736373302023 -0.908349638293848
        0.364479345611439 -0.917420432703305
        0.358036171422893 -0.926485805882573
        0.351404683561074 -0.935542149609919
        0.34458276830493 -0.944585749035173
        0.337568369437474 -0.95361278229087
        0.330359492394054 -0.962619320232526
        0.322954208506376 -0.971601326315002
        0.315350659339254 -0.980554656611933
        0.307547061116611 -0.989475059985285
        0.299541709232905 -0.998358178412107
        0.291332982845686 -1.0071995474756
        0.282919349544582 -1.01599459702755
        0.274299370091575 -1.02473865202926
        0.265471703226937 -1.03342693357784
        0.256435110534783 -1.04205456012501
        0.247188461361659 -1.05061654889503
        0.237730737781151 -1.05910781750855
        0.228061039596992 -1.06752318581891
        0.218178589376631 -1.07585737796729
        0.208082737506763 -1.08410502466264
        0.197772967261781 -1.09226066569257
        0.187248899875635 -1.10031875267046
        0.176510299607073 -1.10827365202429
        0.165557078787743 -1.11611964823202
        0.154389302842133 -1.12385094730794
        0.143007195267877 -1.13146168054425
        0.131411142564443 -1.13894590851133
        0.119601699097784 -1.14629762531988
        0.107579591888101 -1.15351076314746
        0.0953457253074121 -1.16057919703148
        0.0829011856732524 -1.16749674992986
        0.0702472457244347 -1.17425719805024
        0.0573853689644637 -1.18085427644757
        0.0443172138578676 -1.18728168488947
        0.0310446378644383 -1.19353309398782
        0.0175697012961168 -1.19960215159427
        0.00389467098105867 -1.20548248945655
        -0.00997797628075014 -1.21116773013157
        -0.0240455504860727 -1.21665149415037
        -0.0383051454370107 -1.22192740742914
        -0.0527536360054285 -1.22698910891956
        -0.0673876756689993 -1.23183025849059
        -0.0822036943345612 -1.23644454503305
        -0.0971978964643095 -1.24082569477733
        -0.112366259520129 -1.24496747981322
        -0.127704532741082 -1.24886372680036
        -0.14320823626871 -1.25250832585628
        -0.158872660634389 -1.25589523960824
        -0.174692866622459 -1.25901851239402
        -0.190663685522324 -1.26187227959581
        -0.206779719782025 -1.26445077709012
        -0.223035344075109 -1.26674835079612
        -0.239424706791831 -1.26875946630333
        -0.255941731964833 -1.2704787185591
        -0.272580121638545 -1.2719008415951
        -0.289333358690531 -1.27302071827145
        -0.306194710111924 -1.27383339001619
        -0.323157230752958 -1.27433406653702
        -0.3402137675384 -1.27451813548166
        -0.357356964156398 -1.27438117202252
        -0.37457926622294 -1.27391894834073
        -0.391872926922739 -1.27312744298436
        -0.409230013125904 -1.27200285007485
        -0.426642411978293 -1.27054158833589
        -0.444101837961897 -1.26874030991827
        -0.46159984042007 -1.26659590899446
        -0.479127811540819 -1.26410553009648
        -0.496676994789746 -1.26126657617079
        -0.514238493782629 -1.25807671632396
        -0.531803281585987 -1.25453389323343
        -0.549362210432327 -1.25063633019785
        -0.566906021835165 -1.24638253780206
        -0.584425357087283 -1.24177132017246
        -0.601910768124101 -1.23680178079925
        -0.619352728732497 -1.23147332790265
        -0.63674164608386 -1.22578567932165
        -0.654067872568744 -1.2197388669043
        -0.67132171790902 -1.21333324038027
        -0.688493461522135 -1.2065694706973
        -0.705573365110766 -1.19944855280482
        -0.722551685450006 -1.19197180786917
        -0.739418687343083 -1.1841408849069
        -0.756164656715643 -1.17595776182366
        -0.772779913817689 -1.16742474584843
        -0.789254826501491 -1.15854447335443
        -0.80557982354313 -1.14931990905983
        -0.82174540797471 -1.13975434460342
        -0.837742170393933 -1.12985139649236
        -0.853560802217339 -1.11961500342101
        -0.869192108843404 -1.10904942296194
        -0.884627022691626 -1.09815922763233
        -0.899856616083828 -1.08694930034087
        -0.914872113934147 -1.07542482922239
        -0.929664906214571 -1.06359130186958
        -0.944226560163365 -1.05145449897294
        -0.958548832204403 -1.03902048738242
        -0.972623679546199 -1.02629561260574
        -0.98644327143033 -1.01328649076064
        -1 -1
        -1 -1
        -1.08396802079176 -0.903801279087144
        -1.15376634320221 -0.797805787542312
        -1.20806127651684 -0.684790428137221
        -1.24618171506462 -0.567698952788257
        -1.26811958698952 -0.449461842586164
        -1.27447772366242 -0.332831532316949
        -1.26637593937751 -0.220248363719364
        -1.24533009628614 -0.113746717053327
        -1.21311999080154 -0.0149042898468527
        -1.1716603788652 0.0751682413558739
        -1.12288617876032 0.155804067634701
        -1.06865883358767 0.226732695097273
        -1.01069685196347 0.288017505732133
        -0.950530273178122 0.339987023786573
        -0.889476507143117 0.383166763136213
        -0.828633700843642 0.4182165206865
        -0.768887329573193 0.445876052805709
        -0.710925871519104 0.466920453100678
        -0.655261960054074 0.482125335652529
        -0.602256117962338 0.492241118014422
        -0.552140915132525 0.497975236623405
        -0.505044064142878 0.49998093379188
        -0.461009530078872 0.498851248436581
        -0.420016167749182 0.49511695123151
        -0.381993717390794 0.489247333785775
        -0.346836205781482 0.481652951709599
        -0.314412934062453 0.472689607933149
        -0.28457730685154 0.462663031084713
        -0.25717378722189 0.451833847509274
        -0.232043263291553 0.440422563040624
        -0.209027095565081 0.428614363343678
        -0.187970087806248 0.416563612470218
        -0.168722593632119 0.40439798175448
        -0.151141939812241 0.392222178849759
        -0.135093317627192 0.380121272794159
        -0.120450266786249 0.368163628258197
        -0.107094852851408 0.356403472845241
        -0.0949176189617239 0.344883127301242
        -0.0838173757383704 0.333634931149535
        -0.073700879278539 0.322682896665196
        -0.0644824357565486 0.312044123031501
        -0.0560834619729187 0.301730000551361
        -0.048432023873826 0.291747232325268
        -0.0414623692840851 0.282098698134828
        -0.035114466575658 0.272784182574661
        -0.0293335574901682 0.263800986875299
        -0.0240697296457873 0.255144441427517
        -0.0192775122187192 0.246808333792415
        -0.0149154967602551 0.238785264976973
        -0.0109459839809284 0.231066944972134
        -0.00733465651440393 0.223644436980569
        -0.00405027709418956 0.216508358389548
        -0.00106441117966807 0.209649045353458
        0.00164882719010048 0.203056686821774
        0.00411300769059302 0.196721432963645
        0.00634959560730852 0.190633482182036
        0.00837814480736519 0.184783150262555
        0.0102164732438068 0.17916092464989
        0.0118808224033841 0.173757506375171
        0.0133860020444145 0.168563841758839
        0.0147455214957706 0.163571145675588
        0.0159717087062408 0.158770917881819
        0.0170758181492049 0.154154953664074
        0.0180681286034553 0.149715349862531
        0.0189580317489598 0.145444507151152
        0.0197541124376748 0.141335129310631
        0.0204642214249792 0.137380220107733
        0.0210955412773955 0.133573078291479
        0.0216546461072003 0.129907291129821
        0.0221475557243548 0.126376726837536
        0.0225797847408151 0.122975526184784
        0.0229563871115399 0.119698093524412
        0.0232819965501762 0.116539087432966
        0.0235608632152023 0.113493411124326
        0.0237968870239668 0.110556202764717
        0.0239936479172961 0.107722825792627
        0.0241544333658666 0.104988859326204
        0.0242822633810757 0.102350088723197
        0.0243799132674457 0.0998024963439973
        0.0244499343303941 0.097342252556329
        0.0244946727322897 0.0949657070101903
        0.024516286670858 0.0926693802034898
        0.0245167620370159 0.0904499553521108
        0.0244979266939146 0.08830427057268
        0.0244614635051958 0.0862293113818801
        0.0244089222280602 0.0842222035125759
        0.0243417303755774 0.0822802060441764
        0.024261203142612 0.0804007048424039
        0.0241685524806823 0.0785812063018879
        0.0240648953989102 0.0768193313836692
        0.0239512615608771 0.0751128099386948
        0.0238286002405738 0.0734594753076779
        0.0236977866946705 0.0718572591872055
        0.0235596280029479 0.0703041867516866
        0.023414868423882 0.0687983720205877
        0.0232641943079954 0.0673380134603852
        0.0231082386076355 0.0659213898107442
        0.0229475850182759 0.0645468561245898
        0.0227827717832105 0.0632128400119565
        0.0226142951906012 0.0619178380777667
        0.0224426127892043 0.0606604125439872
        0.0222681463467208 0.0594391880469409
        0.0220912845725618 0.0582528486008876
        0.0219123856248649 0.0571001347193464
        0.0217317794198367 0.0559798406859835
        0.0215497697598896 0.0548908119672516
        0.021366636295591 0.0538319427593204
        0.0211826363351276 0.05280217366219
        0.0209980065137897 0.0518004894742198
        0.0208129643348961 0.050825917100641
        0.0206277095925948 0.0498775235699446
        0.0204424256860779 0.0489544141523473
        0.0202572808339336 0.0480557305748415
        0.0200724291966184 0.0471806493276221
        0.0198880119143569 0.0463283800569598
        0.0197041580671608 0.045498164039856
        0.0195209855631007 0.0446892727360646
        0.0193386019604528 0.0439010064133069
        0.0191571052288773 0.0431326928417348
        0.0189765844543585 0.0423836860539115
        0.0187971204922535 0.0416533651667889
        0.0186187865724357 0.0409411332623507
        0.018441648860201 0.0402464163237788
        0.0182657669763056 0.0395686622241743
        0.0180911944792347 0.0389073397650269
        0.0179179793125528 0.0382619377617879
        0.0177461642199608 0.037631964174044
        0.0175757871304757 0.037016945277933
        0.0174068815159593 0.0364164248785706
        0.0172394767230485 0.0358299635603839
        0.0170735982813769 0.0352571379733641
        0.0169092681898343 0.0346975401533598
        0.0167465051824718 0.034150776874638
        0.016585324975541 0.0336164690330392
        0.0164257404970364 0.0330942510581418
        0.0162677621000106 0.0325837703529434
        0.0161113977608332 0.0320846867596446
        0.0159566532634751 0.0315966720502021
        0.0158035323708208 0.0311194094403874
        0.0156520369839341 0.0306525931261605
        0.0155021672901352 0.0301959278412299
        0.0153539219006815 0.0297491284347328
        0.0152072979787888 0.029311919468027
        0.015062291358672 0.0288840348296411
        0.0149188966562371 0.02846521736748
        0.014777107372009 0.0280552185374315
        0.0146369159868373 0.0276537980675672
        0.0144983140508832 0.0272607236371702
        0.0143612922663541 0.0268757705698672
        0.0142258405644191 0.0264987215401783
        0.0140919481767071 0.0261293662928344
        0.0139596037017618 0.0257675013742463
        0.0138287951667994 0.0254129298755439
        0.0136995100850921 0.0250654611866297
        0.0135717355092768 0.0247249107607256
        0.0134454580808671 0.0243910998889136
        0.0133206640762283 0.0240638554842004
        0.0131973394492552 0.0237430098746585
        0.0130754698709787 0.0234284006052194
        0.0129550407663074 0.023119870247717
        0.0128360373481003 0.0228172662187981
        0.0127184446487506 0.0225204406053389
        0.0126022475494486 0.0222292499970207
        0.0124874308072812 0.0219435553257405
        0.0123739790803144 0.0216632217115428
        0.0122618769507931 0.0213881183147801
        0.0121511089465873 0.0211181181942188
        0.0120416595610024 0.0208530981708259
        0.0119335132710625 0.0205929386969813
        0.0118266545543717 0.0203375237308751
        0.0117210679046475 0.0200867406158608
        0.0116167378460179 0.019840479964545
        0.011513648946163 0.0195986355474066
        0.0114117858283818 0.0193611041857478
        0.011311133182655 0.0191277856487881
        0.0112116757757729 0.0188985825547224
        0.0111133984605913 0.0186734002755725
        0.0110162861844746 0.0184521468456687
        0.0109203239969817 0.0182347328736077
        0.0108254970568458 0.0180210714575377
        0.0107317906382977 0.0178110781036309
        0.0106391901367755 0.0176046706476089
        0.0105476810740649 0.0174017691791929
        0.0104572491029084 0.0172022959693561
        0.0103678800111203 0.0170061754002633
        0.0102795597252421 0.0168133338977846
        0.0101922743137702 0.0166236998664787
        0.010106009989986 0.0164372036269439
        0.0100207531144166 0.0162537773554397
        0.00993649019695184 0.0160733550256878
        0.00985320789864334 0.0158958723527632
        0.00977089303320684 0.0157212667389914
        0.00968953256825065 0.0155494772217724
        0.00960911362624955 0.0153804444232522
        0.00952962348528306 0.0152141105017706
        0.00945104957955554 0.0150504191050139
        0.00937337949971455 0.0148893153248051
        0.00929660099298269 0.0147307456534676
        0.00922070196311719 0.0145746579417014
        0.00922070196311719 0.0145746579417014
        0.00773530454125697 0.0116484416098829
        0.00646694729750431 0.00933308577807472
        0.00539224444950198 0.00749568423179227
        0.00448686831069901 0.0060334133266749
        0.00372747181896563 0.0048665006030749
        0.0030926541083692 0.00393286606281656
        0.00256335783486274 0.0031840414682132
        0.00212293940265453 0.00258206074433204
        0.00175706025804709 0.00209708635168586
        0.00145348905090043 0.00170559308609155
        0.00120186799289238 0.00138897439873933
        0.000993474095765711 0.00113246953266793
        0.000820992039096985 0.000924334837661147
        0.000678306959582823 0.00075520146596792
        0.000560320424873556 0.000617575784508377
        0.00046278995985631 0.000505449438166397
        0.00038219091421682 0.000413993954168388
        0.000315598681673962 0.00033932075901701
        0.000260588979833899 0.000278291987240875
        0.00021515387312483 0.000228370867807863
        0.000177631344789695 0.000187503056362418
        0.000146646421210429 0.000154022245373259
        0.000121062077396628 0.000126574883030396
        9.99383800057294e-05 0.000104059979514401
      };
    \end{axis}
  \end{tikzpicture}
  \caption{Image of $q_\mu([0,\infty))$ for $\mu = 1 + i$.}
  \label{fig:image-q-mu}
\end{figure}
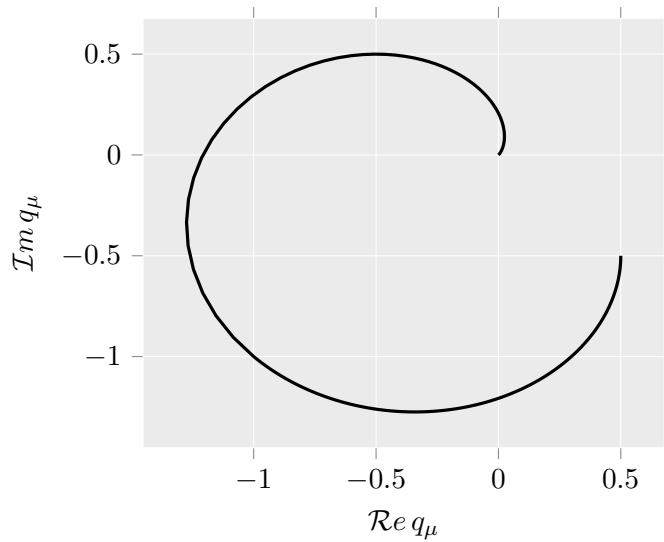

The advantage of the stability criterion is that it is explicitly
computable and has an immediate visual interpretation as the image of
\(\Phi_\infty(\R + i \epsilon)\) for \(\epsilon > 0\).

As an example, we consider
\begin{equation}
  \label{eq:example1}
  V_s(y) = x + 4x \, e^{-2x}.
\end{equation}
In this case we find the resulting curve is shown in
Figure~\ref{fig:example1}. Here we see that it is stable as it is not
crossing the positive real axis.

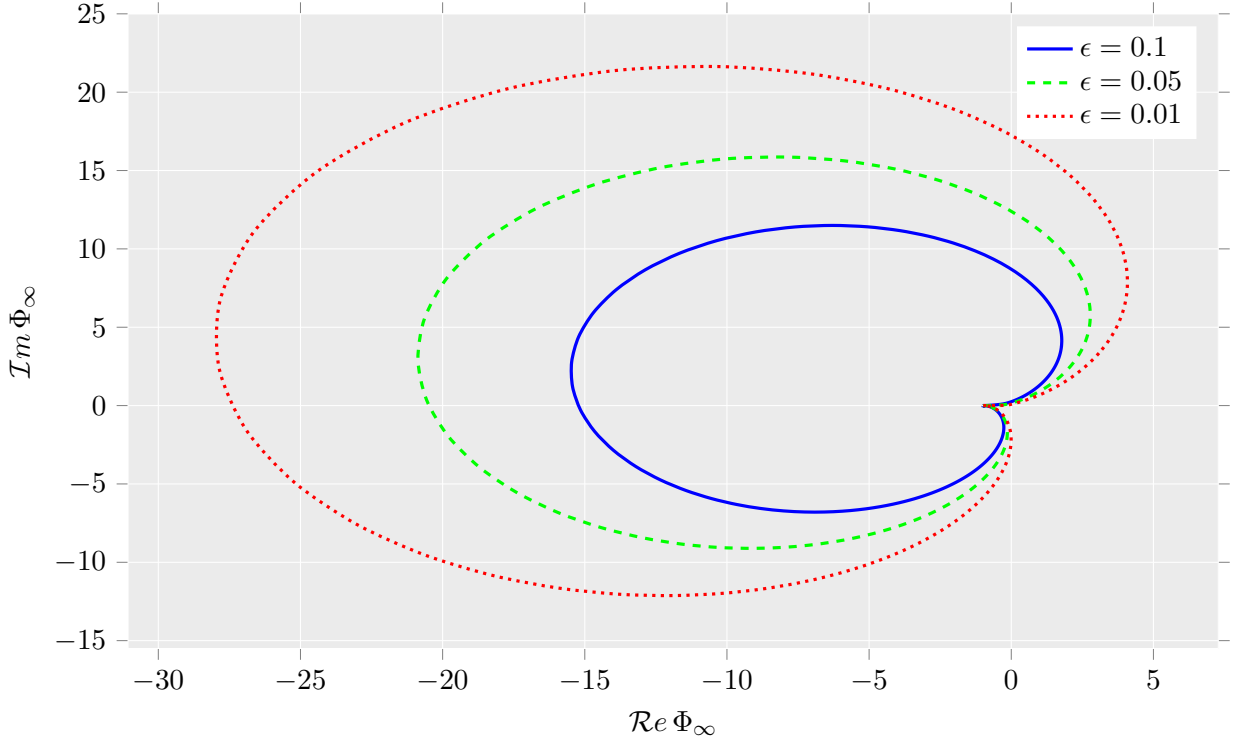
\begin{figure}[htb]
  \centering
  \begin{tikzpicture}
    \begin{axis}[
      xlabel={\(\Re \Phi_\infty\)},
      ylabel={\(\Im \Phi_\infty\)},
      smooth,
      width=\textwidth,
      height=10cm,
      xtick distance=5
      ]
      \addplot [very thick,blue]
      table {ex4-offset1.txt};
      \addlegendentry{\(\epsilon=0.1\)}
      \addplot [very thick, dashed,green]
      table {ex4-offset2.txt};
      \addlegendentry{\(\epsilon=0.05\)}
      \addplot [very thick, dotted,red]
      table {ex4-offset3.txt};
      \addlegendentry{\(\epsilon=0.01\)}
    \end{axis}
  \end{tikzpicture}
  \caption{Image of $\Phi_\infty(\R + i \epsilon)$ for $V_s$ from
    \eqref{eq:example1} with $\epsilon=0.1$, \(\epsilon=0.5\) and \(\epsilon=0.01\).}
  \label{fig:example1}
\end{figure}

As another example, we consider
\begin{equation}
  \label{eq:example2}
  V_s(y) =  \sin(2x)\, e^{-x} + x\, (1-e^{-x})
\end{equation}
In this case we find the resulting curve is shown in
Figure~\ref{fig:example2}. Here we see that unstable modes exists as
the positive real axis is crossed once.

\begin{figure}[htb]
  \centering
  \begin{tikzpicture}
    \begin{axis}[
      xlabel={\(\Re \Phi_\infty\)},
      ylabel={\(\Im \Phi_\infty\)},
      smooth
      ]
      \addplot [very thick,blue]
      table {ex6-offset1.txt};
      \addlegendentry{\(\epsilon=0.1\)}
      \addplot [very thick, dashed,green]
      table {ex6-offset2.txt};
      \addlegendentry{\(\epsilon=0.05\)}
    \end{axis}
  \end{tikzpicture}
  \caption{Image of $\Phi_\infty(\R + i \epsilon)$ for $V_s$ from
    \eqref{eq:example2} with
    $\epsilon=0.1$ and \(\epsilon=0.5\).}
  \label{fig:example2}
\end{figure}
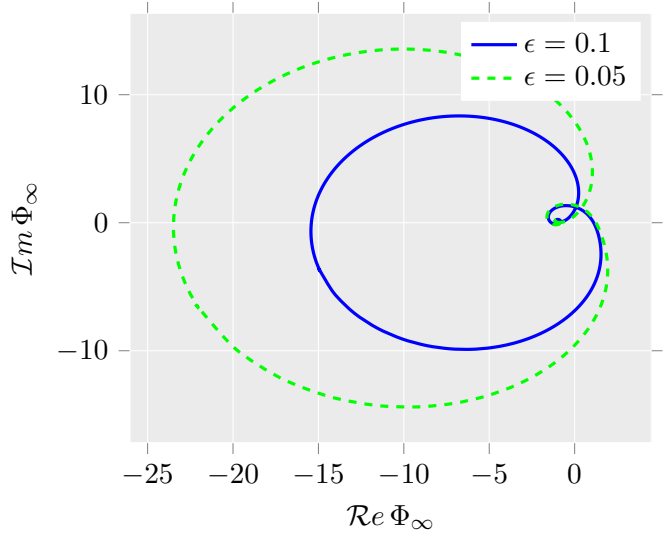

\appendix
\section{Stability for $U_s = 0$} \label{appendixA}

The assumptions of Theorem~\ref{main_thm} are not satisfied in the
case where $U_s$ vanishes identically. On the contrary, one can show
in this setting that any family of solutions of \eqref{LTD} of type
\begin{equation*}
  u_k(t,x,z) =  e^{\lambda_k t} e^{i kx} \hat{u}_k(z), \quad A_k(t,x)
  =  e^{\lambda_k t} e^{i kx}, \quad k \in \R,
\end{equation*}
has  growth rate $(\Re \lambda_k)^+  =O(1)$ as $|k| \rightarrow
+\infty$. We now sketch  the proof of this claim. It builds upon
classical works on the linear stability of Couette or Blasius flows
within Navier-Stokes, {\it cf.} \cite{DrazinReid,Smith79, Was}.

The trick is to differentiate the linearised equation once by \(z\)
and express the result in terms of \(\omega(z) :=
\hat{u}_k'(z)\). Then the linearised evolution
\eqref{eq:simple-linear-system} yields the eigenmode equation
\begin{equation*}
  (\lambda_k  + i k z) \omega -  \omega''  = 0
\end{equation*}
with the boundary conditions
\begin{equation*}
  \omega'(0) = i k |k|,
  \lim_{z\to\infty} \omega = 0
\end{equation*}
and the consistency equation
\begin{equation}
  \label{eq:consistency-simple}
  1 = \int_0^\infty \omega(z)\, \dd z
\end{equation}
by using that \(\omega = \hat{u}_k'(z)\) and
\(\lim_{z \to \infty} \hat{u}_k(z) = 1\).
The idea is then to make a change of variable in the complex plane to get back to the Airy equation
$$ \xi \varphi(\xi) - \pa^2_\xi \varphi(\xi) = 0. $$
We set
\begin{equation*}
  \eta_k := (ik)^{-2/3} \lambda_k, \quad
  \xi := (ik)^{1/3} y, \quad
  W(\xi) := \omega(y)
\end{equation*}
where the roots are chosen with positive real part so that
\begin{equation*}
  (\xi + \eta_k) W - W'' = 0.
\end{equation*}
Using the boundary condition on $\omega'(0)$ and the decay condition
of $\omega$ at infinity, we get
\begin{equation*}
  W(\xi)  = ik |k| (ik)^{-1/3}
  \frac{\Ai(\xi+\eta_k)}{\Ai(\eta_k, -1)}
\end{equation*}
where $\Ai$ is the so-called Airy function of the first kind, and
$\Ai(\cdot, -1)$ denotes its derivative, {\it cf.}
\cite{DrazinReid}. Then the consistency
equation~\eqref{eq:consistency-simple} yields the dispersion relation
\begin{equation} \label{dispersion}
1 = i k |k| \frac{\Ai(\eta_k, 1)}{\Ai(\eta_k,-1)},
\end{equation}
where $\Ai(\cdot,1)$ is the antiderivative of $\Ai$ that vanishes at
$+\infty$, and we recall $\eta_k = (ik)^{-2/3}
\lambda_k$.

We will now use these relations to determine the asymptotic behaviour
of unstable eigenvalues $\lambda_k$ (meaning $\Re \lambda_k > 0$) when
$|k| \rightarrow +\infty$.  We distinguish between three regimes:
$\eta_k$ goes to zero, goes to infinity, or is $O(1)$ as
$|k| \rightarrow +\infty$.

\begin{itemize}
\item If $\eta_k \rightarrow 0$
\end{itemize}
We find
\begin{equation*}
  1 \sim ik |k| \frac{\Ai(0, 1)}{\Ai(0,-1)}
\end{equation*}
where $\frac{\Ai(0, 1)}{\Ai(0,-1)} = 3^{-2/3} \Gamma(1/3)$, see
\cite[equation (A11)]{DrazinReid}. This yields a contradiction.
\begin{itemize}
\item If  $|\eta_k| \rightarrow +\infty$
\end{itemize}
We use the asymptotic expansion given in \cite{DrazinReid}, see
equations (A12)-(A13)-(A14). In the case $k > 0$, we find
\begin{equation*}
  1 \: \sim \: \frac{i k |k|}{\eta_k}
  \frac{(1- \frac{3 a(1)}{2} (\eta_k)^{-3/2})}{(1 - \frac{3 a(-1)}{2}
   (\eta_k)^{-3/2})}
\end{equation*}
where $a(p) = \frac{1}{72} (12 p^2 + 24 p + 5)$. Using the
$\eta_k = (ik)^{-2/3} \lambda_k$ we find \(\Re \lambda_k \le 0\).

\begin{itemize}
\item If $\eta_k \sim O(1)$
\end{itemize}
Then
\begin{equation*}
  \Ai(\eta_k,1) = \Ai(\eta_k,-1) i k |k|.
\end{equation*}
Hence we see that a subsequence of $\eta_k$ should converge to some
$\eta^0$ satisfying $ \Ai(\eta^0,1) = 0$. As $\Re \lambda_k > 0$, one
finds that $-5\pi/6 < {\rm arg}(\eta_k) < \pi/6$ and thus
$-5\pi/6 \le {\rm arg}(\eta^0) \le \pi/6$.  Similarly, in the case
$k < 0$, one should have $-\pi/6 \le {\rm arg}(\eta^0) \le 5 \pi/6$.

\medskip
These two scenarios are excluded by the following proposition, which can be found in \cite{Was}: {\em the function $\Ai(\cdot,1)$ has no zero in the closed sector $-5\pi/6 \le {\rm arg}(\eta) \le 5\pi/6$}. This concludes the proof of spectral stability.

\section*{Acknowledgements}

The authors acknowledge the support of SingFlows project, grant
ANR-18- CE40-0027 of the French National Research Agency (ANR). D. G-V
acknowledges the support of the Institut Universitaire de France.

{\small
\bibliographystyle{siam}
\bibliography{TD_analysis}
}
\end{document}